\documentclass{article}

\usepackage[T1]{fontenc}
\usepackage[utf8]{inputenc}

\usepackage{soul}
\usepackage{cancel}
\usepackage{import}
\usepackage{indentfirst}

\usepackage{graphicx}

\usepackage{subcaption}

\usepackage{amsmath,amsthm,amssymb,amsfonts,amscd, amsbsy,mathtools}
\usepackage{wasysym}

\usepackage{mathrsfs}

\usepackage{empheq}

\usepackage{tikz}
\usetikzlibrary{calc}

\usepackage{ upgreek }

\usepackage{pdfpages}

\usepackage{bussproofs}

\everymath{\displaystyle}
\DeclareMathAlphabet{\mathcal}{OMS}{cmsy}{m}{n}

\usepackage{faktor}

\usepackage[letterpaper,top=2cm,bottom=2cm,left=3cm,right=3cm,marginparwidth=1.75cm]{geometry}

\usepackage{graphicx}
\usepackage[colorlinks=true, allcolors=blue]{hyperref}

\newtheorem{remark}{Remark}[section]
\newtheorem{prop}{Proposition}[section]
\newtheorem{theorem}{Theorem}[section]
\newtheorem{lemma}{Lemma}[section]

\title{Geometry of a class of $\varphi$-minimal surfaces in $\mathbb{H}^2\times \mathbb{R}$}

\author{
  Katianny Rosa Freitas
  \thanks{Departamento de Matemática,
  Universidade de Brasília, Brazil.
  Email: \href{mailto:katiannyrosa.f@gmail.com}{katiannyrosa.f@gmail.com}}
  \and
  António Martínez
  \thanks{Departamento de Geometría y Topología, Universidad de Granada, Granada, Spain.
  Email: \href{mailto:amartine@ugr.es}{amartine@ugr.es}}
  \and
  João Paulo dos Santos%
  \thanks{Departamento de Matemática,
  Universidade de Brasília, Brazil.
  Email: \href{mailto:joaopsantos@unb.br}{joaopsantos@unb.br}}
  }

\date{\today}

\begin{document}
\maketitle

\begin{abstract}
We study a class of $\varphi$-minimal surfaces in the product space $\mathbb{H}^2\times\mathbb{R}$, that is, weighted minimal surfaces whose corresponding weight is determined by a smooth potential function $\varphi$ that depends only on the height coordinate relative to the factor $\mathbb{R}$. Under monotonicity, concavity, and asymptotic conditions on $\varphi$, we obtain a complete classification of connected, properly embedded $\varphi$-minimal surfaces that are invariant under isometries of $\mathbb{H}^2\times\mathbb{R}$ acting trivially on $\mathbb{R}$ and by rotations or parabolic translations on $\mathbb{H}^2$. Finally, we establish several non-existence results for $\varphi$-minimal surfaces in $\mathbb{H}^2 \times \mathbb{R}$, either by using properties of the height function or by applying the tangency principle together with suitable invariant surfaces as barriers. \\ \\
\end{abstract}

\noindent
\begin{it}2020 Mathematics Subject Classification.\end{it} 53A10; 53C42.

\medskip

\noindent
\begin{it}Keywords.\end{it} weighted minimal surfaces, product spaces, invariant surfaces, rotations, parabolic translations.

\section{Introduction}

The study of surfaces with prescribed mean curvature is a classical topic in differential geometry of surfaces in three-dimensional Riemannian manifolds. Important and extensively explored examples within this topic are given by minimal surfaces, which are critical points of the area functional, and surfaces with non-zero constant mean curvature, which are, in turn, critical points of the area functional with fixed volume. A particularly elegant class are given by the weighted minimal surfaces, that is, surfaces that arise as critical points of a weighted area functional whose weight is determined by a potential function $\varphi$. In this case, these surfaces satisfy an Euler-Lagrange equation given by
\begin{equation}
\mathbf{H} = (\nabla \varphi)^\perp, \label{eq:main} 
\end{equation}
where $\mathbf{H}$ is the mean curvature vector and $(\nabla \varphi)^\perp$ denotes the normal part of $\nabla \varphi$. This class generalizes several important families of surfaces with prescribed mean curvature: classical minimal surfaces (when $\varphi$ is constant), translating solitons for the mean curvature flow in the Euclidean space $\mathbb{R}^3$ (when $\varphi(z)=z$), and singular minimal surfaces (when $\varphi(z)=\alpha\log z$). We refer to surfaces satisfying \eqref{eq:main} as $\varphi$-minimal surfaces. 

Translating solitons for the mean curvature flow constitute a special class of solutions to this flow, as a particular class of solitons, i.e., solutions that evolve under isometries of the ambient spaces. For the specific case of translation solitons, the evolution is given by translations in a given direction. Translating solitons have received particular attention due to their role in the singularity theory of mean curvature flow in Euclidean space \cite{hoffman2017notes}. In Euclidean space, the classification of rotationally symmetric translating solitons is given by the bowl solitons and the translating catenoids \cite{clutterbuck2007stability, martin2015topology}, and such classification has been fundamental in understanding the geometry of these objects. Singular minimal surfaces constitute another important class of prescribed mean curvature surfaces. Their geometry, with special attention to invariant surfaces under one-parameter group of isometries and construction of such surfaces, has been investigated in Euclidean space in $\mathbb{R}^3$ for instance, in \cite{lopez2018invariant, lopez2019dirichlet, dierkes2021symmetric, weierstrass-martinez-trivino}.

The classes of examples of weighted minimal surfaces described above share a particular feature on their potential functions: both depend only on the height function with respect to the vector $\vec{e}_3 = (0,0,1) \in \mathbb{R}^3.$ Weighted minimal surfaces in $\mathbb{R}^3$ where the corresponding weight is determined by a potential function that depends only on the height function with respect to a given vector field. As described in \cite{martinez2022equilibrium}, these surfaces are a case of the equilibrium problem of a flexible and inextensible surface in a vertical force field proposed by Poisson \cite{poisson1975surfaces}. Such a class of weighted minimal surfaces has been recently studied, both in relation to invariant surfaces \cite{martinez2022equilibrium}. and with respect to global aspects and behaviors as we can see, for example, in \cite{martinez2022equilibrium, mean-convex, catenary-cylinders, half-space, rosales1, rosales2} .

More recently, the theory of translating solitons to the mean curvature flow has been extended to product spaces such as $\mathbb{H}^2\times\mathbb{R}$ as we can see in \cite{bueno2018translating, de2019translating, lima-pipoli}, for construction of invariant surfaces, uniqueness results in \cite{bueno2020uniqueness}, and in \cite{tinaglia-pipoli} for asymptotic behaviour. In all of the works mentioned above, vertical translating solitons are considered, namely, solitons where the evolution is given by the vertical Killing field the generates the factor $\mathbb{R}$. In this case, such surfaces are also $\varphi-$ minimal surfaces, where $\varphi : \mathbb{H}^2 \times \mathbb{R} \rightarrow \mathbb{R}$ is given by $\varphi(p,z)=z$. As we described above, singular minimal surfaces constitute another important class of prescribed mean curvature surfaces. However, unlike the class of vertical translating solitons, this class in $\mathbb{H}^2 \times \mathbb{R}$ remains largely unexplored. 

The main goal of this paper is to develop a general classification theory for $\varphi$-minimal surfaces in $\mathbb H^2\times\mathbb R$ when the potential depends only on the height function, with focus on surfaces invariant under a one-parameter group of isometries in $\mathbb{H}^2 \times \mathbb{R}$. More precisely, we assume that $\varphi:]a,\infty[\to\mathbb R$, $a\in \mathbb{R}$, satisfies
\begin{itemize}
    \item[H1)] $\dot{\varphi}(z) > 0$ and $\dot{\varphi}(z) < 0$;
    \item[H2)]   $\lim_{s \to \infty}\dot{\varphi}(s) = 0$. 
    \item[H3)] $\lim_{z \to a^+}(z-a) \dot{\varphi}(z) = L \in ] 0, \infty].$
\end{itemize}
\begin{remark}\label{implicações hipotese 3}
    It is not difficult do prove, from condition (H3) that $$\lim_{z \to a^+}\dot{\varphi}(z) = \infty \quad \text{and} \quad \lim_{z \to a^+}\varphi(z) = -\infty.$$ 
\end{remark}
These assumptions are motivated by the singular minimal surface equation, for which $\dot{\varphi}(z)=1/z$. More information about these surfaces in Euclidean spaces appears in \cite{dierkes2021symmetric} and \cite{lopez2018invariant}. 

Regarding the transformations we will consider to generate invariant surfaces, they will be given by one-parameter groups of isometries of $\mathbb{H}^2 \times \mathbb{R}$ that act trivially on $\mathbb{R}$ and by rotations, parabolic translations, or hyperbolic translations on $\mathbb{H}^2$. Such invariance property will provide, for the corresponding to the prescribed mean curvature problem, a three-dimensional system of ordinary differential equations for the generating curves. A detailed analysis of such systems will provide us, in the rotational and parabolic settings, several families of invariant surfaces, and a complete classification within the parabolic and rotational settings. More precisely, we classify the connected, properly embedded $\varphi$-minimal surfaces in $\mathbb{H}^2\times\mathbb{R}$ that are invariant under either parabolic translations or rotations of $\mathbb{H}^2$. We show that, up to ambient isometry, such surfaces are one of the following:
\begin{itemize}
    \item a parabolic winglike surface (Theorem \ref{teo parabolic com minimo});
    \item a parabolic horizontal graph without a minimum point at height (Theorem \ref{teo sin min parabolic});
    \item a rotational bowl (Theorem \ref{teo bowl rot});
    \item a rotational winglike surface (Theorem \ref{rotatinal winglike});
    \item a rotational horizontal graph without a minimum point at height (Theorem \ref{teo sin min rot}).
\end{itemize}
Each of these families exists and provides surfaces that are properly embedded and complete. 

{ An important contribution of this work is the construction of parabolically and rotationally invariant $\varphi$-minimal surfaces in $\mathbb{H}^2\times\mathbb{R}$ whose height function has no interior minimum and which extend regularly to the boundary. These families are essential to the classification, as they account for the case in which the infimum of the height is not attained in the interior. Related rotational examples in Euclidean Space were obtained by Dierkes and Groh \cite[Theorem 7(iii)]{dierkes2021symmetric} or the singular minimal surface equation. The techniques developed in Section 7 also extend to the Euclidean rotational setting, providing analogous examples for a broader class of height-dependent potentials and recovering their examples in the logarithmic case. }

In addition to the classification results, we prove several non-existence theorems. In particular, we show that there exist no closed $\varphi$-minimal surfaces in $\mathbb{H}^2\times\mathbb{R}$, no properly immersed $\varphi$-minimal surfaces contained in a vertical cylinder, and no complete properly embedded $\varphi$-minimal surfaces with compact boundary contained in a solid horocylinder.

The paper is organized as follows. Section \ref{preliminar} characterizes $\varphi$-minimal surfaces as critical points of a weighted area functional and as minimal surfaces in an ambient space with a conformally equivalent metric. Section \ref{simetrias de h2} introduces the isometries of $\mathbb{H}^2$ and the parametrizations of symmetric surfaces. In Section \ref{geral}, we prove general results for $\varphi$-minimal symmetric surfaces. Sections \ref{parabolic section} and \ref{rotational section} are dedicated to the parabolic and rotational cases, respectively. In Section \ref{sin minimo section}, we study surfaces without a minimum height point. Finally, Section \ref{classification-section} presents the complete classification of rotationally and parabolically invariant $\varphi$-minimal surfaces for $\varphi$ satisfying (H1), (H2) and (H3), and Section \ref{non existence section} presents non-existence theorems.

\section{Preliminaries} \label{preliminar}

    Let $\varphi: {{\cal O}\subseteq}\ \mathbb{H}^2\times\mathbb{R} \to \mathbb{R}$ be a differentiable function defined on an open subset ${{\cal O}}$ of $\mathbb{H}^2\times\mathbb{R}$. We aim to investigate the surfaces in $\mathbb{H}^2\times\mathbb{R}$ whose mean curvature vector is given by
    $$ \mathbf{H} = \left( \nabla \varphi \right)^\perp. $$
    We will refer to these types of surfaces as $\varphi$-minimal surfaces or weighted minimal surfaces, since they are minimal surfaces in a conformal metric $\Tilde{g} = e^{\varphi} g$ of ${{\cal O}\subseteq}\ \mathbb{H}^2\times\mathbb{R}$ and critical points of a weighted area functional (see Theorem \ref{minima espaço conforme}).

  \begin{theorem} \label{minima espaço conforme}
      Let $\Sigma$ be a surface in $\mathbb{H}^2\times\mathbb{R}$. Then, the following statements are equivalent 
      \begin{itemize}
          \item[i)] $\Sigma$ is a $\varphi$-minimal surface;
          \item[ii)] $\Sigma$ is a minimal surface in the conformal space $\left({{\cal O}}, e^{\varphi}g \right)$, where $g$ is the product metric in $\mathbb{H}^2\times\mathbb{R}$.
          \item[iii)] $\Sigma$ is a critical point of weighted area functional
          \begin{equation}\label{area functional}
            A(\Sigma) = \int_\Sigma \text{e}^\varphi d\mu, 
          \end{equation}
          for every compactly supported variation of $\Sigma.$
      \end{itemize}
  \end{theorem}
  \begin{proof}
      Let $\Sigma$ be a $\varphi$-minimal surface and consider the conformal space $\left({{\cal O}}, e^\varphi g \right)$. As is  well known, the mean curvature $\tilde{H}$ of $\Sigma$ in the conformal space is given by
      $$ \tilde{H} = \text{e}^{-\frac{\varphi}{2}}\left( H - g\left( \nabla \varphi, \eta\right)\right),$$
      where $\eta$ is the normal field of $\Sigma$ in $\mathbb{H}^2\times\mathbb{R}$ (see, for example, \cite{tarcios-joao} for a computation). Hence, $\Sigma$ is a $\varphi$-minimal surface if and only if it is minimal in the conformal space. This establishes the equivalence (i) $\Leftrightarrow$ (ii).

    We now show that (i) is equivalent to (iii). For that, we follow \cite{arezzo2013conformal}. Let $\Psi: \Sigma \times (-\varepsilon, \varepsilon) \to \mathbb{H}^2\times\mathbb{R}$ be a compactly supported variation of $\Sigma$ such that $\Psi_t(\Sigma)$ is immersed in $\mathbb{H}^2\times\mathbb{R}$ for each $t$ and $\Psi_0(\Sigma) = \Sigma$. Write the variational field as $V(p)=\dfrac{\partial \Psi}{\partial t}(p,0).$ According to \cite{arezzo2013conformal}, the first variation of the weighted area functional \eqref{area functional} is
    $$
    \frac{d}{dt}\bigg|_{t=0} A (\Psi_t(\Sigma)) = \int_\Sigma e^{\varphi} g(\left( \nabla \varphi \right)^\perp - \mathbf{H}, V) \, d\mu.
    $$
    Consequently, $\Sigma$ is a critical point of the functional weighted area  $A$ if and only if $H = g(\nabla\varphi, \eta)$.
  \end{proof}

    Furthermore, according to the second item in \eqref{minima espaço conforme}, we conclude that the $\varphi$-minimal surfaces satisfy the tangency principle (see \cite{eschenburg1989maximum}).

\begin{theorem}[Tangency Principle]\label{thm:tangency-principle}
Let $M_1$ and $M_2$ be two connected $\varphi$-minimal surface in $\mathbb{H}^2 \times \mathbb{R}$ with non-empty boundaries $\partial M_1$ and $\partial M_2$, respectively. Suppose one of the following conditions holds:
\begin{itemize}
    \item There exists an interior point $p \in \operatorname{int}(M_1) \cap \operatorname{int}(M_2)$ where the unit normal vector fields coincide, i.e., $(\eta_1)_p = (\eta_2)_p$, where $\eta_i: M_i \to \mathbb{S}^2$ is the normal vector field of $M_i$; or
    \item There exists a boundary point $p \in \partial M_1 \cap \partial M_2$ where the inward unit normal vector fields coincide, i.e., $(\tilde{x}i_1)_p = (\tilde{x}i_2)_p$, where $\tilde{x}i_i$ is the inward unit normal to $\partial M_i$.
\end{itemize}
Assume that $M_1$ lies locally around p at one side of $M_2$. Then $M_1$ and $M_2$ coincide in some neighbourhood of $p$.
\end{theorem}

The Tangency Principle will be useful especially in the last section where we prove existence and non-existence theorems for $\varphi$-minimal surface. 

From now on, we focus on $\varphi$-minimal surfaces for which $\varphi$ depends only on the height function of $\mathbb{H}^2\times\mathbb{R}$, i.e.,
    $$\varphi(p,z) = \varphi(z),$$
for all $(p,z)$ in the product space $\mathbb{H}^2\times I$, with $I \subset \mathbb{R}$ a real interval. In this case, the mean curvature equation becomes
    $$ H = \dot{\varphi}(z)\langle \eta, \partial_z \rangle,$$
where $\eta$ denotes the unit normal vector field on the surface.

Observe that, in this case, if the $\varphi-$minimal surface $\Sigma$ is a vertical graph of the function $u: \Omega \subset \mathbb{H}^2\to \mathbb{R}$, $\Omega$ an open subset, then  $u$ satisfies the equation
 \begin{equation} \label{grafico}
     \text{div}\left( \frac{ \nabla u }{\sqrt{1 + |\nabla u |^2}}\right) = \frac{\overset{.}{\varphi}(u)}{\sqrt{1+|\nabla u|^2}}.
 \end{equation}

Solving the PDE \eqref{grafico} for arbitrary surfaces is a problem of considerable complexity. A standard approach in the theory of surfaces in homogeneous spaces is to restrict the analysis to surfaces that are invariant under one-parameter groups of isometries of the ambient space. This is the subject of the next section.

\section{ Isometries of the Hyperbolic Plane and Symmetric Surfaces}\label{simetrias de h2}

This section introduces the main geometric objects and symmetries used throughout the classification of $\varphi-$minimal surfaces in $\mathbb{H}^2\times\mathbb{R}$. The fundamental strategy is to consider surfaces that are invariant under the action of isometry groups of $\mathbb{H}^2$, trivially extended to the product space $\mathbb{H}^2\times\mathbb{R}$. This method, widely used in the literature (see \cite{bueno2018translating}, \cite{de2020hypersurfaces}, \cite{lima-pipoli}, \cite{de2019translating}), reduces the problem of classifying surfaces in $\mathbb{H}^2\times\mathbb{R}$  to the analysis of one-dimensional generating curves.

In the upper half-space
\begin{equation*}
    \mathbb{H}^2 = \left\{ w = x + iy \in \mathbb{C}; y > 0\right\},
\end{equation*}
with the metric $\text{d}s^s = \frac{\left( \text{d}x^2 + \text{d}y^2  \right)}{y^2}$, the orientation-preserving isometries are, precisely, the Möbius transformations (see \cite{anderson2005hyperbolic}). These isometries are divided into three classes, depending on their fixed point: rotational, parabolic and hyperbolic isometries. Each of these classes fixes a family of curves in $\mathbb{H}^2$.

Parabolic isometries are Möbius transformations with exactly one fixed point on the ideal boundary $\partial_\infty\mathbb{H}^2$ and act by translations along horocycles centered at this point. Rotational isometries have one fixed point in the interior of $\mathbb{H}^2$ and another outside the model. Their orbits are geodesic circles centered at this interior point. Hyperbolic isometries have two distinct fixed points on the ideal boundary
$\partial_\infty \mathbb{H}^2$, and their orbits are equidistant curves (hypercycles) to the unique geodesic connecting these two boundary points.

\begin{figure}[ht] 
    \centering
    
    \begin{subfigure}[b]{0.31\textwidth} \centering 
    %%%%%%%% ROTACAO
    \resizebox{\linewidth}{!}{
    \begin{tikzpicture}[scale=0.9] 
        % axes 
        \draw[->] (-2,0) -- (2,0) node[right] {$x$};
        \draw[->] (0,-0.2) -- (0,3.5) node[above] {$y$};
        % boundary 
        \node[left] at (-2,0) {$\partial_\infty\mathbb H^2$}; 
        % fixed point
        \fill (0,1.4) circle (2pt); \node[right] at (0,1.4) {$p$};
        % geodesic circles
        \draw[red,thick] (0,1.4) circle (0.35); \draw[red,thick] (0,1.4) circle (0.60); \draw[red,thick] (0,1.4) circle (0.90); \draw[red,thick] (0,1.4) circle (1.25); 
        \end{tikzpicture}}
        \caption{Rotational isometries.} \label{fig:rotation} \end{subfigure} 
        \hfill 
        \begin{subfigure}[b]{0.31\textwidth}
        \centering
        %%%%%%%%%% HIPERBOLICA 
        \resizebox{\linewidth}{!}{\begin{tikzpicture} 
        % axes 
        \draw[->] (-2,0)--(2,0) node[right]{$x$}; 
        \draw[->] (0,0)--(0,3.5) node[above]{$y$}; 
        % boundary 
        \node[left] at (-2,0) {$\partial_\infty\mathbb H^2$}; 

        %fixed point
          \fill (-1.5,0) circle (2pt) node[below] {$p$};
          \fill (1.5,0) circle (2pt) node[below] {$q$};

        % geodesic
        \draw[dashed,thick] (-1.5,0) arc (180:0:1.5); 
        \node at (0.3,1.08){$\gamma$};
        
        % hypercycles 
        \pgfmathsetmacro{\A}{cosh(0.25)} \pgfmathsetmacro{\B}{sinh(0.25)}
        % upper family
        
        \draw[red,line width=0.9 pt,samples=150,smooth,domain=-13:195] plot ({1.5*\A*cos(\x)},{1.5*(\B+\A*sin(\x))}); \pgfmathsetmacro{\A}{cosh(0.5)} \pgfmathsetmacro{\B}{sinh(0.5)} 
        
        % upper family
        \draw[red,line width=0.9pt,samples=150,smooth,domain=-25:205] plot ({1.5*\A*cos(\x)},{1.5*(\B+\A*sin(\x))}); 
        \foreach \d in {0.25,0.50,0.75} { \pgfmathsetmacro{\A}{cosh(\d)} \pgfmathsetmacro{\B}{sinh(\d)} 
        % lower family
        \draw[red,line width=0.9pt,samples=150,smooth,domain=asin(\B/\A):180-asin(\B/\A)] plot ({1.5*\A*cos(\x)},{1.5*(-\B+\A*sin(\x))}); } \end{tikzpicture}}
        \caption{Hyperbolic isometries.} 
        \label{fig:hyperbolic} 
        \end{subfigure} \hfill \begin{subfigure}[b]{0.31\textwidth} \centering 
        %%%%%%%%% PARABOLIC
        \resizebox{\linewidth}{!}{\begin{tikzpicture} 
        % axes 
        \draw[->] (-2,0) -- (2,0) node[right] {$x$}; 
        \draw[->] (0,-0.2) -- (0,3.5) node[above] {$y$}; 
        \node[left] at (-2,0) {$\partial_\infty\mathbb H^2$};
        % fixed point 
        \fill (0,0) circle (2pt);
        \node[below] at (0,0) {$q$}; % horocycles
        \draw[red,thick] (0,0.35) circle (0.35); 
        \draw[red,thick] (0,0.65) circle (0.65);
        \draw[red,thick] (0,0.95) circle (0.95);
        \draw[red,thick] (0,1.30) circle (1.30);
        \end{tikzpicture}}
        \caption{Parabolic isometries.} \label{fig:parabolic} \end{subfigure} \caption{Invariant families of curves associated with one-parameter groups of isometries of $\mathbb{H}^2$.} \label{fig:isometries} 
        \end{figure}
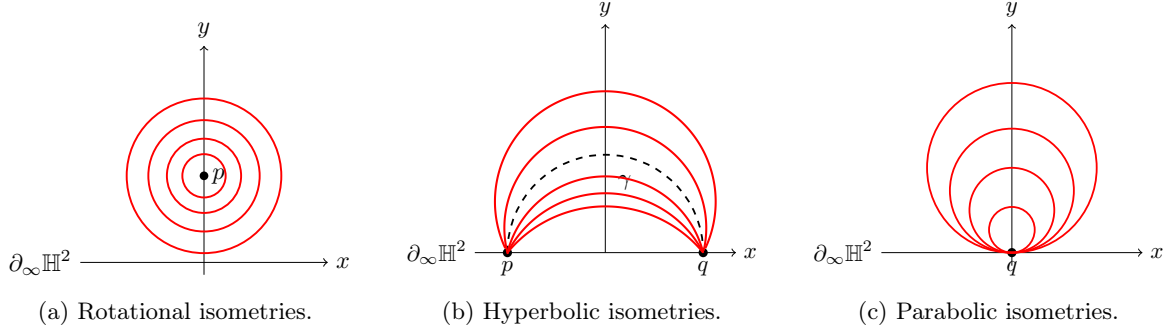

These isometries of $\mathbb{H}^2$ extend naturally to isometries of $\mathbb{H}^2\times\mathbb{R}$, by acting trivially on the $\mathbb{R}-$coordinate
\begin{equation*}
    \widetilde{T}(p,z) = (T(p),z), \quad T \in \text{Isom}^+(\mathbb{H}^2). 
\end{equation*}

\iffalse
 Surfaces invariant under an isometry group of $\mathbb{H}^2\times\mathbb{R}$ are foliated by curves (orbits of the group). Thus, the problem of classifying two-dimensional surfaces in $\mathbb{H}^2\times\mathbb{R}$ reduces to the study of one-dimensional generating curves, described by ODEs systems. This strategy was fundamental, for example, in the classification of translating solitons in $\mathbb{H}^2\times\mathbb{R}$ by Bueno \cite{bueno2018translating}. 
 \fi

Parametrizations for surfaces invariant under rotations, hyperbolic translations, and parabolic translations in \( \mathbb{H}^n \times \mathbb{R} \) are based on the general construction of hypersurfaces presented in \cite{ManfioTojeiro2011, joeri}. In order to write down the parametrization of invariant surfaces in each case, let us consider the Lorentzian model of $\mathbb{H}^2 \times \mathbb{R}$. Let $\mathbb{L}^4$ be the Lorentzian 4-space given by $\mathbb{R}^4$ endowed with the Lorentzian metric
$$
g_{\mathbb{L}} = -dx_0^2 + dx_1^2 + dx_2^2 + dx_3^2.
$$
In this context, the Lorentzian model of $\mathbb{H}^2 \times \mathbb{R}$ is given by
\begin{equation}
    \mathbb{H}^2 \times \mathbb{R} = \left\{ (x_0,x_1,x_2,x_3) \in \mathbb{L}^4 \, , \, -x_0^2 + x_1^2 + x_2^2 = -1 \right\}. \label{eq:lorentzian-model}
\end{equation}

 An invariant surface in $\mathbb{H}^2\times\mathbb{R}$ will be generated by a profile curve
$\gamma:I\to \mathbb{H}^2\times \mathbb{R}$ given by 
\begin{equation}\label{profile curve}
  \gamma(t) = (\cosh x(t), 0, \sinh x(t), z(t)),  
\end{equation} parametrized by the arc length $t$, that is,
\begin{equation}\label{pca}
    (x'(t))^2 + (z'(t))^2 = 1,
\end{equation}
composed with an 1-parameter group of Lorentzian transformations that will provide the rotation, hyperbolic translations and parabolic translations.

\iffalse
The profile curve is determined by the functions $x(s)$ and $z(s)$. In this work, we restrict our attention to families of surfaces invariant under parabolic and rotational isometries.  For convenience, the algebraic computations are carried out using the Lorentz–Minkowski model of $\mathbb{H}^2\times\mathbb{R}$, while the Poincaré disk model is used to visualize the resulting surfaces.
\fi

\subsection{Surfaces invariant under parabolic translations}

Let $\Sigma$ be a surface in $\mathbb{H}^2\times \mathbb{R}$ invariant under parabolic translations. Thus, up to rigid motions, we can parametrize $\Sigma$ by composing the curve \eqref{profile curve} with the extensions of the one-parameter group of parabolic isometries given by
\begin{equation}
P_u=
\begin{pmatrix}
1+\dfrac{u^2}{2} & u & -\dfrac{u^2}{2} \\
u & 1 & -u \\
\dfrac{u^2}{2} & u & 1-\dfrac{u^2}{2}
\end{pmatrix}, u \in \mathbb{R}, \label{eq:parabolic}    
\end{equation}
which has $P=(1,0,1)$ as fixed point. Therefore, the parametrization is given by
\begin{equation}\label{parametrization parabolic}
    \psi(u,t) = \widetilde{P}_u (\gamma(t)) = \left(\cosh x(t) + \frac{u^2}{2}e^{-x(t)}, u e^{-x(t)},\sinh x(t) + \frac{u^2}{2}e^{-x(t)},z(t)\right).
\end{equation}
A standard computation gives a unit normal vector field for $\psi$ as 
\begin{equation*}
    \eta(u,t) = \left(\left(\sinh x(t) - \frac{e^{-x(t)}u^2}{2}\right)z'(t), -ue^{-x(t)}z'(t), \left(\cosh x(t) - \frac{e^{-x(t)}u^2, -x'(t)}{2}\right)z'(t), -x'(t)\right),
\end{equation*}
with corresponding (non-normalized) mean curvature given by
\begin{equation*}
    H(u,t) = x''(t)z'(t)-x'(t)z''(t)+z'(t).
\end{equation*}
We conclude that $\Sigma$ is a $\varphi$-minimal surface if and only if the functions $x$ and $z$ are solutions of the ordinary differential equation system
\begin{equation} 
    \left\{\begin{array}{ll}
         x '(t) & = \cos \theta(t)  \\
          z'(t)& = \sin \theta(t) \\
          \theta '(t)& = \dot{\varphi}(z(t))\cos \theta(t) + \sin \theta(t).
    \end{array}\right.
\end{equation}

\subsection{Surfaces invariant under rotations}

Now we consider $\Sigma$ as a surface in $\mathbb{H}^2\times\mathbb{R}$ invariant under rotational isometries. Then, up to rigid motions, $\Sigma$ is parametrized by the composition of $\gamma$ with the extensions of the rotations
\begin{equation}
    R_u=
\begin{pmatrix}
1 & 0 & 0 \\
0 & \cos u & \sin u \\
0 & -\sin u & \cos u
\end{pmatrix}, u \in \mathbb{R},
\end{equation}
that fixes the point $P=(1,0,0).$ Similarly to the parabolic case, the parametrization is given by
\begin{equation} \label{parametrizacao rot}
   \psi(u,t) =  \widetilde{R}_u (\gamma(t)) = (\cosh x(t) , \sinh x(t) \sin u, \sinh x(t) \cos u , z(t)). 
\end{equation}

A unit normal vector field is given by
$$\eta(u,t) =  \left( z'(t) \sinh x(t) , - z'(t)\cosh x(t) \sin u, - z'(t)\cosh x(t) \cos u , -x'(t) \right),$$
with corresponding (non-normalized) mean curvature
$$H(u,t) = x''(t)z'(t) - x'(t)z''(t) - z'(t)\coth x(t).$$

From the formulas above, we have that $\Sigma$ is a $\varphi$-minimal surface if and only if the functions $x$ and $z$ are solutions of the ordinary differential equation system
\begin{equation} \label{prob 1}
    \left\{\begin{array}{rl}
        x'(t) = & \cos \theta(t)  \\
        z'(t)  = &\sin\theta(t) \\
        \theta'(t)  = & \dot{\varphi}(z(t))\cos\theta(t) - \sin\theta(t)\coth x(t).
    \end{array}\right.
\end{equation}

\subsection{Surfaces invariant under hyperbolic translations}

Finally, we consider $\Sigma$ as a surface in $\mathbb{H}^2\times\mathbb{R}$ invariant under hyperbolic translations. Up to rigid motions, $\Sigma$ is parametrized by the composition of $\gamma$ with the extensions of the hyperbolic translations
\begin{equation}
    \widetilde{H}_u=
\begin{pmatrix}
\cosh u & \sinh u & 0 \\
\sinh u & \cosh u & 0   \\
0 & 0 & 1
\end{pmatrix}, u \in \mathbb{R},
\end{equation}
that fixes the geodesic $l(s) = (0, \cosh s, \sinh s) $. In this case, the parametrization is given by
\begin{equation} \label{parametrizacao hiperbolica}
   \psi(u, t) =  \widetilde{H}_u (\gamma(t)) = (\cosh u \cosh x(t) , \sinh u \cosh x(t) , \sinh x(t), z(t)). 
\end{equation}

Similarly to the previous cases, a unit normal vector field is given by
$$\eta(u,t) =  \left( z'(t) \sinh x(t) \cosh u ,  z'(t) \sinh x(t) \sinh u,  z'(t)\cosh x(t) , -x'(t) \right),$$
with corresponding (non-normalized) mean curvature
$$H(u,t) = x''(t)z'(t) - x'(t)z''(t) - z'(t)\tanh x(t).$$

From the formulas above, we have that $\Sigma$ is a $\varphi$-minimal surface if and only if the functions $x$ and $z$ are solutions of the ordinary differential equation system
\begin{equation} \label{prob 2}
    \left\{\begin{array}{rl}
        x'(t) = & \cos \theta (t)  \\
        z'(t)  = &\sin\theta (t) \\
        \theta'(t)  = & \dot{\varphi}(z(t))\cos\theta(t) - \sin\theta(t)\tanh x(t).
    \end{array}\right.
\end{equation}

\section{General results for \texorpdfstring{symmetric $\varphi$-minimal surfaces}{phi-minimal} } \label{geral}

Having established the parametrizations for each symmetry class, we now turn to the study of properties that are common to all \(\varphi\)-minimal symmetric surfaces where the function $\varphi$ satisfies the hypothesis (H1), (H2) and (H3). This section is dedicated to proving initial results about such $\varphi$-minimal surfaces. As in the previous section, to classify a $\varphi$-minimal symmetric surface it is enough to determine the behaviour of the functions $x$ and $z$ of the profile curve \eqref{profile curve}. These functions are solutions of
\begin{equation}\label{system 1}
  \left\{  \begin{array}{ll}
    x' & =  \cos \theta, \\
     z' & = \sin \theta, \\
     \theta' &= \dot{\varphi}(z) \cos\theta -\left(\tanh x\right)^i \sin\theta,
\end{array}\right.
\end{equation}
where we omitted the parameter $t$ to simplify the notation, and the function $(\tanh)^i$ is defined in each case as in the table \ref{tabela} below.  
\begin{table}[ht]
\centering
\begin{tabular}{|l|l|l|}\hline
    $i$ & Function & Case \\ \hline
    $-1$ & $\coth(x)$ & Rotational\\ \hline
    $0$ & $-1$ & Parabolic    \\  \hline  
    $1$ & $  \tanh(x)$  &  Hyperbolic\\\hline
\end{tabular}
\caption{symmetric cases}\label{tabela}
\end{table}
By the existence and uniqueness theorem for ODEs, given an initial condition, there exists a unique solution of (\ref{system 1}) in a maximal interval $(t_{\textnormal{min}},t_{\textnormal{max}})$, which implies that there exist $\varphi$-minimal surfaces in each of these cases. 

{The function $x$ has a natural geometric interpretation. Indeed, considering the profile curve $\gamma$, as in \eqref{profile curve},
in the Lorentzian model, the hyperbolic distance between its projection onto $\mathbb H^2$ and the point $P=(1,0,0)$ is given by
$$ d_{\mathbb H^2}\big(\pi_{\mathbb H^2}(\gamma(t)),P\big)=|x(t)|.
 $$
Thus, $|x|$ measures the hyperbolic distance from the profile curve to the reference point $P$, while the sign of $x$ provides its position in one of the half-spaces determined by the geodesic plane over the geodesic $\alpha(s)=(\cosh(s),\sinh(s),0)$, which meets the geodesic $\beta(s)=(\cosh(s), 0, \sinh(s))$ orthogonally at $P$. Describing the function $x$ is important for determining whether the surface is a vertical graph, whether it lies in a half-space of $\mathbb{H}^2\times\mathbb{R}$, or whether it approaches the boundary of $\mathbb{H}^2$ or is bounded. Moreover, the angle function in any one of these cases is given by:
$$\nu = \langle\eta, \partial_z\rangle = - \cos \theta,$$
where $\eta$ is a normal field. Thus, if $x$ is monotone, we conclude that the surface is a graph. Equally important is understanding the behaviour of the function $z$, especially regarding critical points and limits. 

\begin{remark}
    Since $x'=\cos\theta$ and $z'=\sin\theta$, the angle $\theta$ determines the monotonicity of the coordinate functions $x$ and $z$. Moreover, the system is invariant under the transformation
    $$ (x,z,\theta)\longmapsto(-x,z,\pi-\theta).$$

 This transformation reverses the sign of $x'$ while preserving the sign of $z'$, in the rotational and hyperbolic cases. 

Consequently, up to reflection with respect to $x=0$, it is sufficient to consider initial conditions satisfying

$$
z_0> a
\qquad\text{and}\qquad
\theta_0\in\left[-\frac{\pi}{2},\frac{\pi}{2}\right].
$$
\end{remark}

 We now begin the qualitative analysis of the solutions considering the case $0<\theta_0<\pi/2$, for which both coordinate functions are initially increasing, since $x'(0)=\cos\theta_0>0$ and $z'(0)=\sin\theta_0>0$. The next lemma shows that the monotonicity of $z$ persists throughout the maximal interval of existence and describes the possible behaviour of $x$ according to the sign of its initial value.

\begin{lemma} \label{crescimento de z}
    Let $x$, $z$ and $\theta$ be solutions of (\ref{system 1}), with the initial condition:
    \begin{equation*}
        x(0)= x_0, \quad z(0) = z_0, \quad \text{and}\quad \theta(0) = \theta_0,
    \end{equation*}
    where $0 \leq \theta_0 < \frac{\pi}{2}$. Then, $z$ is strictly increasing {in $[0,t_{\textnormal{max}})$}. Moreover, the following statements hold {in $[0,t_{\textnormal{max}})$}:
    \begin{itemize}
        \item  In the rotational and hyperbolic cases, if $x_0 \geq 0$ then $x' > 0$.
        \item  If $x_0 < 0$, then $x$ has at most one critical point. 
    \end{itemize}
\end{lemma}
\begin{proof} 
    Suppose, by contradiction, that $t_0$ is the first critical point of $z$. Observe that $z$ is increasing in a sufficiently small right neighbourhood of $t=0$. Indeed, if $\theta_0=0$, then $z'(0)=0$ and $z''(0)>0$. If $\theta_0>0$, then $z'(0)>0$. Therefore, if $t_0>0$ is the first critical point of $z$, it cannot be a local minimum. However,
    $$z''(t_0) = \dot{\varphi}(z(t_0))\cos^2\theta(t_0) > 0,$$
    which is a contradiction. 

    We now consider the behaviour of $x$.  First, in the rotational and the hyperbolic cases, assume that $x_0 \geq 0$ . Suppose, by contradiction, that $t_1$ is the first critical point of $x$. Thus, since $x'(0) > 0$,  $x''(t_1) \leq 0$. On the other hand, as $x$ increases until $t_1$ and $x_0 > 0$, then $x(t_1) > 0$. Consequently,   
    $$x''(t_1) = \sin^2 \theta(t_1) \left( \tanh x(t_1)\right)^i > 0.$$
    That is also a contradiction. 
    
    Finally, consider $x_0 < 0$. In this case, if $x$ has a critical point, it must be unique. We first consider the rotational and hyperbolic cases, corresponding to $i=-1$ and $i=1$, respectively. Suppose, by contradiction, that $t_1$ and $t_2$ are critical points of $x$ with $t_1 < t_2$.  We may choose $t_1$ to be the first
critical point of $x$. Hence $t_1$ is a local maximum, since $x'(0) > 0$, and 
    $$ x''(t_1) =  \sin^2 \theta(t_1)\left( \tanh x(t_1)\right)^i < 0. $$
    That implies that $x(t_1) < 0$.
    
Thus, $x$ decreases in the interval $[t_1, t_2]$, implying that $x(t_2) < x(t_1) < 0 $ and $ x''(t_2) \geq 0$. However,
    $$x''(t_2) = (\tanh x(t_2))^i\sin^2 \theta < 0,$$
 which is a contradiction. Therefore,  $x(t)$  admits at most one critical point.
    
    In the parabolic case, every critical point $t_0$ of $x$ satisfies
    $$x''(t_0)= - \sin^2\theta(t_0) < 0.$$
Hence every critical point of $x$ is a strict local maximum.
Consequently, $x$ cannot have more than one critical point.
    Therefore, the result follows. 
\end{proof}

Thus, by the last lemma, we conclude that every $\varphi$-minimal surface is a vertical graph after a positive time and the height function does not have a maximum point. Observe that if $\Sigma$ is a vertical graph of a function $u$,  then by (\cite{de2020hypersurfaces}, Equation (23)), we have that $u$ satisfies
\begin{equation} \label{edo graph}
    u'' = \left(\dot{\varphi}(u) - (\tanh r)^i u'\right)\left(1+u'^2\right).
\end{equation}

We already know that the solution of this ODE exists and it is unique for each initial condition. Moreover, suppose that $u'(0) = 0$; then $u''(0) = \dot{\varphi}(u(0)) > 0$ and $r = 0$ is a minimum point of $u$. We can prove, similarly to Lemma \ref{crescimento de z}, that $u'(r) > 0$ for all $r > 0$. In this context, we prove the following proposition, motivated by Proposition 4.2 of \cite{martinez2022equilibrium}. 

\begin{prop} \label{comparar}
Let $u_1$ and $u_2$  be solutions of (\ref{edo graph})
with initial conditions
$$u_i(0) = u_{i,0}, \quad u_i'(0) = v_{i,0}.$$
and $\dot{\varphi}_1, \, \dot{\varphi}_2:(a, \infty)\to \mathbb{R}$ let be smooth functions such that 
$$\dot{\varphi}_1(u_1) > \dot{\varphi}_2(u_2).$$
{If $v_{2,0} = 0$ and $v_{1,0} \geq 0$. Then 
$$ u'_1 > u'_2.$$}
\end{prop}
\begin{proof} 
  Define
$$
d(r)=u_1'(r)-u_2'(r).
$$
If $v_1>0$, then $d(0)>0$. If $v_1=0$, then 
$$
d(0)=0,\qquad
d'(0)=\dot{\varphi}_1(u_1(0))-\dot{\varphi}_2(u_2(0))>0.
$$
Thus, in both cases, $d>0$ for sufficiently small $r>0$.

Suppose, by contradiction, that $r_1>0$ is the first
point where $d(r_1)=0$. Then $d'(r_1)\leq0$.
However, since $u_1'(r_1)=u_2'(r_1)$, equation
\eqref{edo graph} gives
$$
d'(r_1)
=
\left(1+u_1'(r_1)^2\right)
\left[
\dot{\varphi}_1(u_1(r_1))-\dot{\varphi}_2(u_2(r_1))
\right]>0.
$$
This contradiction proves that
$u_1'(r)>u_2'(r)$ for every $r>0$
in the common interval of existence.
\end{proof}

Moreover, knowing about the maximal interval of the solution helps us to understand its behaviour, and for this reason we prove the following proposition.

\begin{prop}\label{maximal interval}
    If $z$ is a strictly increasing function, the solution $(x, z, \theta)$ of (\ref{system 1}) is defined for all positive time in the parabolic and hyperbolic cases. Moreover, if $z$ increases and $\liminf_{t \to t_{\max}^+} |x| > 0$, then the solution is also defined for all positive time in the rotational case. 
\end{prop}
\begin{proof} Let  $(x, z, \theta)$ be a solution of (\ref{system 1}) with initial conditions $x(0)=x_0$, $z(0)=z_0$ and $\theta(0)=\theta_0$, defined on a maximal interval $(t_{\textnormal{min}}, t_{\textnormal{max}})$. Observe that in the hyperbolic and parabolic cases, the function $(\tanh(x))^i$  is bounded for all $x$.
    Since $z$ is strictly increasing, and   $\ddot{\varphi}  < 0$, the function $\dot{\varphi}(z(t))$ strictly decreases. Thus, given an initial condition $z(0) = z_0 \in (a, \infty)$, we have that $\dot{\varphi}\left(z(t)\right) < \dot{\varphi}(z_0)$, for all $t$ in $(0,t_{\textnormal{max}})$. Moreover, we know that the solution $y(t) = (x,z,\theta)$ satisfies the integral equation
    \begin{equation*}
        y(t) = y_0 + \int_0^t f(y),
    \end{equation*}
    where $f(y) = (\cos \theta, \sin \theta, \cos\theta \dot{\varphi}(z) - \sin \theta \left( \tanh x)^i\right)$ and $y_0 = y(0).$ Then
    \begin{align*}
        \parallel f(y) \parallel^2 & = 1 + \left(\cos \theta \dot{\varphi}(z) - \sin\theta\left(\tanh x\right)^i\right)^2 \\
        & \leq 1 + \left(\dot{\varphi}(z) + 1 \right)^2 \leq L_1,
    \end{align*}
    with $L_1: = 1 + \left(\dot{\varphi}(z_0) + 1\right)^2$. Then, the solution $y$ satisfies
    $$ \parallel y(t) \parallel \leq \parallel y_0 \parallel 
 + \int_0^t L_1 = \parallel y_0 \parallel + t L_1.$$
 Suppose, by contradiction, that the $t_{\max} < \infty$. Since $z(t) \geq z_0>a$, the solution defined in $[0,t_{\max})$ is contained in the compact set
 $$\left\{ y \in \mathbb{R}\times (a, \infty)\times \mathbb{R}; \parallel y(t) \parallel \leq  \parallel y_0 \parallel + t_{\max}L_1 \right\}.$$
 However, by Corollary 2 in Section 2.4 of \cite{perko2013differential}, $t_{\max} = \infty$.  Now, suppose that $\liminf_{t\to t_{\max}^+}|x| > 0$, 
then the function $\coth(x)$ is bounded, and analogously, we prove that the solution is defined for all positive time.
\end{proof}

The combination of these results provides the necessary foundation for understanding each case of invariant surfaces. In this paper, we describe in details the surfaces invariant under parabolic translations and rotations.  The hyperbolic-invariant setting is geometrically richer and requires more involved analysis than the rotational and parabolic settings, so we treat it separately and in more detail in \cite{hyperbolic-FMS}.

\section{Parabolic Case} \label{parabolic section}

Let $\Sigma$ be a surface in $\mathbb{H}^2\times \mathbb{R}$ foliated by horocycles curves. As mention in Section \ref{simetrias de h2}, the surface $\Sigma$ is determined by the functions $x$ and $z$ of the profile curve \eqref{prob 1}. Theses functions are solutions of \eqref{ode parabolic 2}, 
\begin{equation}\label{ode parabolic 2}
    \left\{\begin{array}{ll}
         x ' & = \cos \theta  \\
          z' & = \sin \theta \\
          \theta '& = \dot{\varphi}(z)\cos \theta + \sin \theta.
    \end{array}\right.
\end{equation}
By the existence and uniqueness Theorem, there exists a unique solution of (\ref{ode parabolic 2}) for each given initial condition {$(x_0,z_0,\theta_0)$}. We will analyse the behaviour of the solution for positive and negative times, which will enable us to classify the parabolic $\varphi$-minimal surfaces.

Note that the system \eqref{ode parabolic 2} can be written as a two-dimensional system depending only on the functions $z$ and $\theta$:
\begin{equation} \label{ode parabolic 2d}
    \left\{\begin{array}{ll}
          z' & = \sin \theta \\
          \theta '& = \dot{\varphi}(z)\cos \theta + \sin \theta.
    \end{array}\right.
\end{equation}
From now on, we will study the solutions of \eqref{ode parabolic 2d}.

\begin{lemma} \label{lema 3}
    Let  $\theta$ and $z$ be solutions of \eqref{ode parabolic 2d} with the initial condition
    \begin{equation} \label{initial cond parabolic}
        z(0) = z_0, \quad \theta(0) = 0.
    \end{equation}
    Then, the solution is defined for all positive time, and both $z$ and $\theta$ are increasing functions. Moreover, 
    $$\lim_{t \to \infty}z(t) = \infty.$$
\end{lemma}
\begin{proof}
    Note that $z'(0) = 0$ and $z''(0) > 0$, then, $z$ increases near $t = 0$. By Lemma \ref{crescimento de z} and Proposition \ref{maximal interval}, the solution is defined for all positive time and $z' > 0$. 

    Observe also that $\theta'(0) = \dot{\varphi}(z) > 0$. Suppose, by contradiction, that there exists $t_0 > 0$ such that $\theta'(t_0) = 0$. Hence, since $\sin \theta(t_0)>0$, it follows that  $\cos \theta(t_0) < 0$. Moreover, $\theta''(t_0) \leq 0$. However,
    $$\theta ''(t_0) = \cos \theta(t_0) \sin \theta(t_0) \ddot{\varphi}(z) > 0,$$
    which is a contradiction. 

Therefore, both $z$ and $\theta$ are increasing functions, with $0<\theta<\pi$, which implies that $\lim_{t \to \infty} \theta(t) = \theta_L$. Suppose by contradiction that $z$ is bounded from above, then there exists the limit $\lim_{t \to \infty} z(t) = z_L$. On the one hand, since $\theta$ as limit, we conclude that $\lim_{t \to \infty} z'(t) = \sin(\theta_L)$ must be zero, so $\theta_L = \pi.$ On the other hand from \eqref{ode parabolic 2d} we have $\lim_{t \to \infty} \theta'(t) = - \dot{\varphi}(z_L)<0$, which is a contradiction with the fact that $\theta'(t)>0, \, \forall t>0$.
\end{proof}

Another important part of characterizing this type of surface is to describe the behaviour of the function $x$. Describing the behaviour of $x$ provides insights into the distance between the surface and the vertical axis, as well as the convexity of the surface. In this context, the previous lemma is also useful to determine the behaviour of $x$,as described in the following proposition.

\begin{prop} \label{prop x comportamento}
    Let $x$, $z$ and $\theta$ solutions of (\ref{ode parabolic 2}) with initial conditions
  \begin{equation}\label{icp} {x(0)=x_0, \quad z(0)=z_0, \quad \theta(0) =0.}\end{equation}
    Then, $x$ has a unique critical point, which is a maximum, and 
    $$\lim_{t \to \infty} x(t) = -\infty$$
\end{prop}
\begin{proof}
    By Lemma \ref{crescimento de z}, we know that if $x$ has a critical point, it must be unique. So, it is enough to prove that there exists a critical point of $x$. Firstly, observe that in a small neighbourhood of $t = 0$ the function $x$ is increasing. Suppose, by contradiction, a the critical point does not exist. In this case, $\cos \theta > 0$ for all $t > 0$. Then, 
    $$\theta '  = \dot{\varphi}(z)\cos\theta + z' > z'.$$
    This inequality implies 
        $$\theta > z - z_0,$$
    for all $t > 0$. Since $\theta$ is bounded by the proof of Lemma \ref{lema 3}, we conclude that there exists a positive constant $c_1$ such that $c_1 > z$, resulting in a contradiction, by the Lemma \ref{lema 3}.
   Therefore, $x$ has a unique critical point $t_0$. Moreover, the critical point of $x$ must be a maximum point.

    By hypotheses, the function $\dot{\varphi}(z(t))$ decreases. Then,  there exists a positive constant $L$ such that $\dot{\varphi}(z(t)) < L$, for $t>0$ large enough. Since $\theta$ is an increasing function with $\theta(t_0) = \pi / 2$ we have $x' < 0$ for $t > t_0$. Consequently, $\dot{\varphi}(z) x' > x' L$ and
    \begin{align*}
        \theta' & > Lx'+z',
    \end{align*}
    which implies that, for all $t>t_0$,
     \begin{align*}
        \theta > Lx + z +c_1,
    \end{align*}
    where $c_1$ is a constant. Hence, since $\theta$ is bounded and $\lim_{t\to \infty} z = \infty $, we have that $\lim_{t\to \infty} x = - \infty$. 
\end{proof}

\begin{remark}
    By this proposition, if the initial condition $x_0 > 0$, we conclude that there exists a $t_1 > t_0$ such that $x(t_1)  = 0$, that is, the surface touches the vertical axis at this time.
\end{remark}

To finish the analysis about the solutions of (\ref{ode parabolic 2})-(\ref{icp}), it remains to study what happen in a negative time. Hence, consider $t<0$ and define the functions

\begin{equation*}
    \tilde{x}(t) = x(-t), \quad \tilde{z}(t) = z(-t), \quad  \text{and} \quad \tilde{\theta}(t)= \theta(-t)+\pi,
\end{equation*}
Note that $(\tilde{x}, \tilde{z}, \tilde{\theta})$ are solutions of the ODE system

\begin{equation} \label{ode parabolic negative time 1}
    \left\{\begin{array}{lll}
         \tilde{x}' & = &\cos \tilde{\theta}  \\
         \tilde{z}' & = &\sin \tilde{\theta} \\
         \tilde{\theta}' & = & \dot{\varphi}(\tilde{z}) \cos\tilde{\theta} + \sin \tilde{\theta},
    \end{array}\right.
\end{equation}
with the initial conditions 
{\begin{align}\label{initial cond para neg}
    \tilde{x}(0) = x_0, \quad \tilde{z}(0) = z_0, \quad \tilde{\theta}(0) = \pi.
\end{align}

These functions are well defined, and there exists a unique solution to (\ref{ode parabolic negative time 1})-(\ref{initial cond para neg}).

\begin{prop}\label{tempo negativo prop}
    Let $\tilde{x}$, $\tilde{z}$ and $\tilde{\theta}$ solutions of (\ref{ode parabolic negative time 1})-(\ref{initial cond para neg}). Then, the solution is defined for all positive time. Moreover, in  $(0,\infty)$,  $\tilde{x}$ is decreasing,  $\tilde{z}$ is increasing and $\tilde{\theta}$ has an unique critical point $t_0$, $\theta(t_0)>\frac{\pi}{2}$,  which is a minimum. Moreover,
    \begin{equation*}
        \lim_{t\to \infty}\tilde{x}(t) =- \infty, \quad \lim_{t\to \infty}\tilde{z}(t) = \infty, \quad\text{and} \quad 
        \lim_{t\to \infty} \tilde{\theta} = \pi.
    \end{equation*}
\end{prop}
\begin{proof}
    Arguing as in Lemma \ref{crescimento de z}, we can show that in $(0,t_{max})$, $\tilde{x}$  is decreasing and $\tilde{z}$ is increasing. Moreover, analogous to Proposition \ref{maximal interval}, we prove that the solution is defined for all positive time.  Concerning $\tilde{\theta}$, we have $\tilde{\theta}'(0)= - \dot{\varphi}(z_0)<0$. Therefore $\tilde{\theta}$ is a decreasing function near $t=0$. We claim that $\tilde{\theta} \in (\pi/ 2, \pi)$ for all $t>0.$ In fact, suppose by contradiction that there exists a first $t_* > 0$ such that $\tilde{\theta}(t_*) = \pi / 2$. On the one hand, since $\tilde{\theta}(0)=\pi$, we have $\tilde{\theta}'(t_*) \leq 0$. On the other hand
    $$
    \tilde{\theta}'(t_*) = \dot{\varphi}(\tilde{z}) \cos \tilde{\theta}(t_*) + \sin \tilde{\theta}(t_*) = 1,
    $$
    which is a contradiction. Now we show that $\tilde{\theta}$ has at most one critical point. Suppose, by contradiction, we have two consecutive critical points $t_0<t_1$ for $\tilde{\theta}$. For any of those $t_i$ we have $\sin \tilde{\theta}(t_i)>0$, because $\tilde{z}'>0$, and therefore $\cos \tilde{\theta}(t_i)<0$. Consequently,
    $$\tilde{\theta} ''(t_i) = \cos \tilde{\theta}(t_i) \sin \tilde{\theta}(t_i) \ddot{\varphi}(z) > 0,$$
    which is a contradiction. We conclude that $\tilde{\theta} \in (\pi/ 2, \pi)$ and has at most one critical point. Therefore,  $\tilde{\theta}$ is a strictly monotone function on $[T,\infty)$, for $T>0$ sufficiently large. Since it is bounded, there exists the limit $\lim_{t \to \infty} \tilde{\theta} (t) = \tilde{\theta}_L$, with $\tilde{\theta}_L \in [\pi/2, \pi]$.
   \\ 
Now we show that there exists $t_0>0$ so that $\tilde{\theta}'(t_0)=0$. In fact, suppose by contradiction that $\tilde{\theta}'<0$ for all $t>0$. Then $z'' = \cos(\tilde{\theta}) \tilde{\theta}' >0$ and $z$ is an unbounded function. In this case, $0=\lim_{t\rightarrow \infty} \dot{\varphi} \cos(\tilde{\theta}) \leq -\lim_{t\rightarrow \infty} \sin(\tilde{\theta}) = -\sin(\tilde{\theta}_L) < 0$, which is a contradiction. So, we conclude that $\tilde{\theta}'(t)>0$ in $(t_0,\infty)$. 
    \\
The next step is to prove that $\tilde{z}$ is unbounded. Suppose, by contradiction, that $z$ is bounded. Since $z$ is increasing in $(0,\infty)$ we have $\lim_{t \to \infty} \tilde{z}(t) = l$. Also, if $z'>0$ and $z'' <0$ in $(t_0,\infty)$, we have $0 =\lim_{t\rightarrow \infty}z'(t)$ and $\lim_{t\rightarrow \infty}\tilde{\theta}(t)=\pi$. Thus, $\lim_{t\rightarrow \infty}\tilde{\theta}'(t)=-\dot{\varphi}(l) < 0$, which is impossible. Therefore, $\tilde{z}$ is unbounded.
\\
Now we obtain the value of $\tilde{\theta}_L$. Since both limits $\lim_{t \to \infty} \tilde{\theta}(t)$ and $\lim_{t \to \infty} \tilde{\theta}'(t)$ exist, we have $\lim_{t \to \infty} \tilde{\theta}'(t) = \sin \tilde{\theta}_L = 0$, that is, $\tilde{\theta}_L = \pi$.
\\
Finally, as $x'<0$ and $x''<0$ in $(t_0,\infty)$, we deduce that $\lim_{t\rightarrow \infty} x(t)=-\infty$ which concludes the proof.
\end{proof}
Therefore, we can prove the existence of a $\varphi$-minimal surface with properties similar to a winglike translating. For this reason, we will call it \textit{Parabolic Winglike}.

\begin{theorem} \label{teo parabolic com minimo}
    There exists a family of properly embedded $\varphi$-minimal surface in $\mathbb{H}^2\times\mathbb{R}$ that is invariant under a one-parameter group of parabolic isometries of $\mathbb{H}^2$. The surface has the topology of $\mathbb{R}^2$ and is a bigraph. Moreover, each of its graphical components has unbounded hyperbolic the horizontal coordinate \(x\) of the generating curve, the height function is unbounded, and its generating curve is of winglike type.
\end{theorem}
\begin{proof}
Combining Lemma \ref{lema 3}, Proposition \ref{prop x comportamento}, and \ref{tempo negativo prop}, the result of the theorem follows.
\end{proof}

\begin{figure}[ht]
  \centering
  \begin{minipage}[b]{0.45\textwidth}
    \includegraphics[scale=0.75]{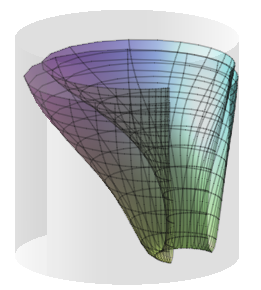}
    
  \end{minipage}
  \hfill
  \begin{minipage}[b]{0.45\textwidth}
    \includegraphics[scale=0.5]{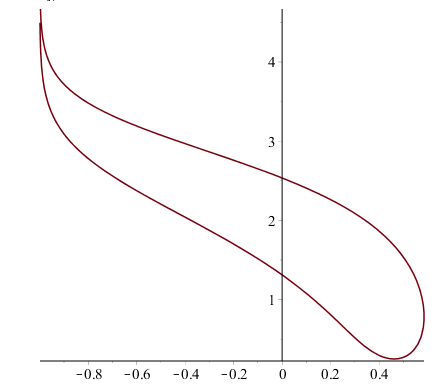}
    
  \end{minipage}

  \caption{$\varphi$-minimal winglike for $\varphi(z) = \ln(z)$}
\end{figure}

Observe that in the Parabolic case, when $\theta(0) = 0$, there not exist a $\varphi$- minimal surface with proprieties similar to a Bowl. However,  when $\varphi-$minimal surface is a translating soliton, that is, when $\varphi(z) = z$, we have there exists a surface foliated by horocycles with similar proprieties to Rotational Bowl. That surface is described by Lima and Pipoli in \cite{lima-pipoli}. In our case, we have that the functions
    \begin{equation*}
        x(t) = \frac{\sqrt{2}}{2}t+x_0, \quad z(t) = - \frac{\sqrt{2}}{2}t + z_0, \quad \text{and} \quad \theta(t) = -\frac{\pi}{4},
    \end{equation*}
    give us this surface. 

\section{Rotational Case} \label{rotational section}

%The study of rotationally invariant surfaces is a classical topic in differential geometry and has been studied by many authors in various settings, including translating solitons, singular minimal surfaces, and weighted minimal surfaces.  Relevant contributions include the classification of rotational translating solitons in $\mathbb{R}^3$ and in $\mathbb{H}^2\times\mathbb{R}$, see \cite{gama2020translating} and  \cite{bueno2018translating}, as well, the study of rotational $\varphi$-minimal surface in Euclidean spaces in \cite{martinez2022equilibrium}.

As we shall see, the rotational case is a bit different from the others since a singularity appears in the equation that describes the surface. For this reason, we will begin our analysis assuming that the rotational, $\varphi-$minimal surface is a $(f_r, u)-$graph, as defined in \cite{de2020hypersurfaces}, where $f_r$ are the geodesic spheres. From Equation (23) of \cite{de2020hypersurfaces}, the mean curvature is given by:
$$ H = \coth(r) \rho + \rho ', $$
with $$\rho = \frac{u'}{\sqrt{1 + (u')^2}}.$$
Then, the function $u$ is the solution of the second order differential equation
\begin{equation}\label{edo 1}
    u'' = \left( \dot{\varphi}(u) - \coth(r) u'\left)\right(1+(u')^2\right).
\end{equation}

By the Existence and Uniqueness for ODEs, this equation with the initial condition
$u(r_0) = u_0 \in (a,\infty)$ and $u'(r_0) = v_0$ has a unique solution. However, equation (\ref{edo 1}) is not well defined for $r = 0$. We will prove that this type of singularity can be removed using ideas similar to those in \cite{martinez2022equilibrium}.

Initially, observe that
\begin{align*}
    \left(\frac{u'\sinh(r)}{\sqrt{ 1 + (u')^2}}\right)' & = \left[(u''\sinh(r) + \cosh(r) u') \sqrt{1+(u')^2} - u'\sinh(r)\left(\frac{u'u''}{\sqrt{1+(u')^2}} \right) \right]\frac{1}{1+(u')^2}\\
     & = \frac{u'' \sinh (r) + u'\cosh (r) }{\sqrt{1+(u')^2}} - \frac{(u')^2u'' \sinh(r)}{\left(1+(u')^2\right)^{3/2}} \\
     & = \frac{u''\sinh(r) }{\left(1+(u')^2\right)^{3/2}}+\frac{u'\cosh(r)}{\sqrt{1+(u')^2}}. 
\end{align*}
Suppose $r> 0$, then multiply the equation (\ref{edo 1}) by $\frac{\sinh(r)}{\sqrt{1+(u')^2}}$ to obtain
\begin{align*}
   \frac{u''\sinh(r)}{\left(1+(u')^2\right)^{3/2}} = \frac{\dot{\varphi}(u)\sinh(r)}{\sqrt{1+(u')^2}} - \frac{u'}{\sqrt{1+(u')^2}},
\end{align*}
which implies that 
\begin{equation}\label{edo 2}
    \left(\frac{u'\sinh(r)}{\sqrt{1+(u')^2}}\right)' = \frac{\dot{\varphi}(u)\sinh(r)}{\sqrt{1+(u')^2}}.
\end{equation}

\begin{lemma}\label{y0=0} {Let $u=u(r)$, $r\in]0,\delta[$ be}  a solution of \eqref{edo 2} {and suppose that there exist $\lim_{r\to 0}u(r) = u(0)\in ]a,\infty[$ and } $\lim_{r\to 0} u'(r)=L_0\in[-\infty,\infty]$. Then $L_0=0$.
\end{lemma}
\begin{proof}
Taking $\epsilon$, $0<\epsilon<\delta$, we have 
$$ \frac{u'(r)\sinh(r)}{\sqrt{1+(u'(r))^2}} -  \frac{u'(\epsilon)\sinh(\epsilon)}{\sqrt{1+(u'(\epsilon))^2}} = \int_\epsilon^r\frac{\dot{\varphi}(u(s))\sinh(s)}{\sqrt{1+(u'(s))^2}} ds.$$
Thus, letting $\epsilon\to 0$, 
$$ \frac{u'(r)}{\sqrt{1+(u'(r))^2}} =\frac{ 1}{\sinh(r)}\int_0^r\frac{\dot{\varphi}(u(s))\sinh(s)}{\sqrt{1+(u'(s))^2}} ds,$$
and by the L'H\^{o}pital rule,
$$  \lim_{r\to 0} \frac{u'(r)}{\sqrt{1+(u'(r))^2}} = \lim_{r\to 0}\frac{\dot{\varphi}(u(r))}{\sqrt{1+u'(r)^2}} \lim_{r\to 0}\tanh r=0,$$
which concludes the proof.
\end{proof}
By Theorem 2 of \cite{serrin-removesingular}, the solution of (\ref{edo 2}) has no non-removable, isolated singularities. Therefore, from Lemma \ref{y0=0}, we are able to investigate the solutions of (\ref{edo 2}) with the initial condition:
\begin{equation}\label{cond inicial}
    u(0) = u_0 \in (a,\infty),  \quad u'(0) = 0.
\end{equation}

Initially, define the functions $g: \mathbb{R} \to (-1, 1) \subset \mathbb{R}$ and $f:(a, \infty)\times \mathbb{R}\to \mathbb{R}$ such that:
   \begin{equation*}
       f(x,y) = \frac{\dot{\varphi}(x)}{\sqrt{1+y^2}} \quad\text{and}\quad g(y) = \frac{y}{\sqrt{1+y^2}}.
   \end{equation*}

\begin{prop}
There exists $\delta > 0$ such that the operator  $ J: C^1[0, \delta] \to C^1[0, \delta]$, given as
\begin{equation*}
    Ju(s) = \int_0^s\frac{\sinh t}{\sinh s}f(u,u') dt,
\end{equation*}
satisfy $|Ju(s)| < 1$, for all $u \in C^1[0, \delta]$.
\end{prop}
\begin{proof}
    Note that $\dot{\varphi}(u)$ is a continuous function in a compact interval $[0, \delta]$. Then, there exists $M > 0$ such that
$$\frac{\dot{\varphi}(u)}{\sqrt{1 + u^{'^2}}} < M.$$
Thus, 
$$ \left|\int_0^s \frac{\sinh t}{\sinh s} \frac{\dot{\varphi}(u)}{\sqrt{1 + u^{'^2}}} \right| \leq \int_0^s \left| \frac{\sinh t}{\sinh s} M \right| \leq \tanh\left(\frac{s}{2}\right) M < \frac{\delta}{2}M.$$

Choose $\delta$ sufficiently small so that $\frac{\delta}{2}M < \delta $ and the proposition holds.
\end{proof}

Define the operator $S: C^1[0, \delta] \to C^1[0, \delta]$ by
   \begin{equation*}
       (Su)(r) = \int^r_0 g^{-1}\left(\int^s_0 \frac{\sinh t}{\sinh s} f(u,u') dt \right)ds +u_0.    \end{equation*}
The last proposition verifies that $S$ is well defined. Analogously to the proof of Proposition 2 in \cite{lopez2018invariant}, we obtain the following proposition.

\begin{prop} \label{prop existencia}
  The problem (\ref{edo 2})-(\ref{cond inicial}) has a unique solution $u \in C^2([0,\delta])$, for some $\delta > 0$, which depends continuously on the initial condition and
    $$u''(0) = \frac{\dot{\varphi}(u_0)}{2}.$$
\end{prop}
\begin{proof} 
    Note that $u$ is a solution of (\ref{edo 1})-(\ref{cond inicial}) only if $u$ is a fixed point of $S$. 

    Observe also that there exists an $\varepsilon > 0$ such that the functions $g^{-1}$  and $f$ are Lipschitz. Then
    \begin{align*}
        \left|(su)(r) - (Sv)(r)  \right| & c_1\leq \int_0^r\int^s_0\left|\frac{\sinh(t)}{\sinh(r)}\right|\left|f(u,u') - f(v,v')\right|dsdt\\
        &\leq c_2\int_0^r\int_0^s \left(\parallel u-v\parallel_\infty + \parallel u'-v'\parallel_\infty\right) \frac{\sinh(t)}{\sinh(r)} \\
        &= c_2\left( 2\ln \cosh (r/2)\right) \left(\parallel u-v\parallel_\infty + \parallel u'-v'\parallel_\infty\right).
    \end{align*}
    Similarly, we have
    \begin{equation*}
        \left| (Su)'(r) - (Sv)'(r) \right| \leq c_3   \left(\parallel u-v\parallel_\infty + \parallel u'-v'\parallel_\infty\right) \frac{\cosh(s) -1}{\sinh (s)},
    \end{equation*}
    where $c_i$ are positive constants. 
    Thus, $S$ is a contraction and, by the Banach Fixed Point Theorem, there exists a solution of \eqref{edo 1} near $ r = 0$. Moreover,
    \begin{align*}
        \lim_{r\to 0} u''(r) = \dot{\varphi}(u_0) - \lim_{r\to 0}u'\coth r.
    \end{align*}
    By L'Hôpital's rules,
    \begin{align*}
        \lim_{r \to 0} u'(r) \coth r & =  \frac{ \lim_{r \to 0}(u' \sinh r + u'' \cosh r)}{\lim_{r \to 0} \cosh r}\\
        & = \lim_{r \to 0} u''.
    \end{align*}
    Hence, 
    $$\lim_{r \to 0}u''(r) = \frac{\dot{\varphi}(u_0)}{2}.$$
    On the other hand, 
    \begin{align*}
        u''(0) & = \lim_{t \to 0^+} \frac{u'(t)}{t} \\
        & = \lim_{t \to 0} \dot{\varphi}(u) \frac{\tanh t}{t} - \lim_{t \to 0}\frac{u''\tanh t}{t\left(1+u'^2\right)}\\
        & = \frac{\dot{\varphi}(u_0)}{2},
    \end{align*}
    which implies that $u \in C^2\left([0,\delta]\right)$, for some $\delta > 0$ small. 
\end{proof}

The vertical graph $u$ can be written as a rotationally symmetric surface, parametrized as in \eqref{parametrizacao rot}. Moreover, its properties are determined by the behaviour of the functions $x$ and $z$, solutions of \eqref{prob 1},
\begin{equation*}
    \left\{\begin{array}{rl}
        x' = & \cos \theta  \\
        z'  = &\sin\theta \\
        \theta'  = & \dot{\varphi}(z)\cos\theta - \sin\theta\coth x.
    \end{array}\right.
\end{equation*}

Furthermore, by Proposition \ref{prop existencia}, the graph touches the rotation axis at height $u(0) = u_0$. Then,
\begin{align} \label{initial cond 2}
x(0) = x_0 \geq 0 & & z(0) = u_0 & & \theta(0) = 0.
\end{align}
Let us study the solutions of the problem (\ref{prob 1}) with initial conditions (\ref{initial cond 2}). 

{\begin{lemma} \label{lemma 1 rot}
    Consider $x$, $z$ and $\theta$ solutions of (\ref{prob 1})-(\ref{initial cond 2}). Then, they are defined for any positive time, $x$ and $z$ are increasing functions in $(0,\infty).$ Moreover
    \begin{equation*}
              \lim_{t \to \infty} x = \infty, \qquad  \lim_{t \to \infty} z = \infty, \qquad \lim_{t \to \infty}\theta =0.
    \end{equation*}
\end{lemma}}
\begin{proof}
  Lemma \ref{lemma 1 rot} and Proposition \ref{maximal interval} show us that $x$ and $z$ are increasing functions and that the solution is defined for all positive time. Consequently, we have that $0<\theta<\frac{\pi}{2}$ in $(0,\infty)$ and therefore
   $$
   x' + z' \geq 1,
   $$
   which implies that
   $$
   \lim_{t \to \infty} \left( x(t) + z(t) + c_0 \right) \geq \infty.
   $$
   The inequality above shows us that at least one of the functions $x$ and $z$ is unbounded. Suppose, by contradiction, that $z$ is a bounded function. Since $x$ increases, there exists $t_0>0$ such that
   $$
   \coth x(t) \leq \coth x(t_0), \,\, t>t_0
   $$
   Therefore, since $z'>0$ we have
   $$
   \theta' \geq \dot{\varphi}(z) \cos \theta - c_1 z',
   $$
   where $c_1 = \coth x(t_0)>0.$ Then, considering the integral
   $$
   \displaystyle \int_{t_0}^t \dot{\varphi}(z) \cos \theta ds =  x \dot{\varphi}(z) - \int_{t_0}^t \left( x \sin \theta \ddot{\varphi}(z) \right) ds + c_2 > \dot{\varphi}(z) + c_2,
   $$
   we deduce that
   $$
   x< \dfrac{c_1 z + \theta + c_3}{\dot{\varphi}(z)},
   $$
   which implies that $x$ is bounded, a contradiction. Therefore, $z$ is unbounded. On the other hand, using the inequalities $\coth x \geq 1$ and $\dot{\varphi}(z_0) \leq \dot{\varphi}(z)$, we have that
   $$
   \theta' \leq \dot{\varphi}(z_0) x' - z'.
   $$
   Consequently,
   $$
   x \geq \dfrac{\theta + z - z_0}{\dot{\varphi}(z_0)},
   $$
   which implies that $\lim_{t \to \infty} x(t) = \infty$. Concerning $\theta$, suppose first that the set of critical points of $\theta$ is contained in $(0,t_c)$. Therefore, $\theta$ is monotone for $t>t_c$ and, since it is bounded, there exists the limit $\lim_{t \to \infty} \theta(t) = \theta_L$. The behaviours of $x$ and $z$ imply that the limit $\lim_{t \to \infty} \theta'(t)$ also exists. Since both limits $\lim_{t \to \infty}\theta(t)$ and $\lim_{t \to \infty}\theta'(t)$ exist, we have $\lim_{t \to \infty}\theta'(t)=-\sin \theta_L=0$, and then, $\theta_L=0$. Suppose now the set of critical points of $\theta$ is unbounded. Take any sequence of critical points $\left\{ t_n \right\}$ for $\theta$ such that $\lim_{n \to \infty} t_n = \infty$. Since 
   $$
   0 = \theta'(t_n) = \dot{\varphi}(z) \cos \theta(t_n) - \sin \theta(t_n) \coth x,
   $$
   we conclude that $\lim_{n \to \infty} \theta(t_n) = 0$. Since $\left\{ t_n \right\}$ is arbitrary, we conclude that $\lim_{t_* \to \infty} \theta(t_*) = 0$, for the critical points $t_*$ of $\theta$. If $\theta$ is monotone in an interval $(t_c, \infty)$, the previous argument shows that $\theta_L = 0$. Otherwise, there are local maximum points $t_M$ arbitrarily large. Since each $t_M$ is a critical point, we have $\lim_{t_M \to \infty}\theta(t_M) = 0$,  therefore $\lim_{t \to \infty}\theta(t) = 0$.
\end{proof}
Combining the previous lemma with Proposition \ref{prop existencia}, we prove the existence of a \textit{$\varphi$-minimal rotational Bowl}.

\begin{theorem} \label{teo bowl rot}
There exists a family of properly embedded, rotational $\varphi$-minimal surface in $\mathbb{H}^2\times\mathbb{R}$, which is an entire vertical graph over $\mathbb{H}^2$. The surface touches the vertical axis orthogonally at a single point, reaches a global minimum height at this point, and has unbounded height. 
\end{theorem}
\begin{proof}
 The existence of functions $x$ and $z$ satisfying (\ref{prob 1}), with the initial conditions
 \begin{equation}\label{cond inicial bowl}
     x(0) = 0, \quad z(0) = z_0, \quad \theta(0) = 0,
 \end{equation}
is guaranteed by the Proposition \ref{prop existencia}. Furthermore, by the regularity of the graph, $x$, $z$ and $\theta$ are at least functions $C^1$. Note that $\Tilde{x}(t) = -x(-t)$, $\Tilde{z}(t) = z(-t)$ and $\Tilde{\theta}(t) = -\theta(-t)$ are also solutions of (\ref{prob 1})-(\ref{initial cond 2}). Therefore, the curve $\lambda$ is symmetric about the vertical axis and we will only consider the case where $t\geq 0$.  By Lemma \ref{lemma 1 rot}, we conclude that this surface is a vertical graph with unbounded height. 
\end{proof}

\begin{figure}[ht]
    \centering    \begin{subfigure}[t]{0.50\textwidth}
        \centering
        \includegraphics[width=\linewidth]{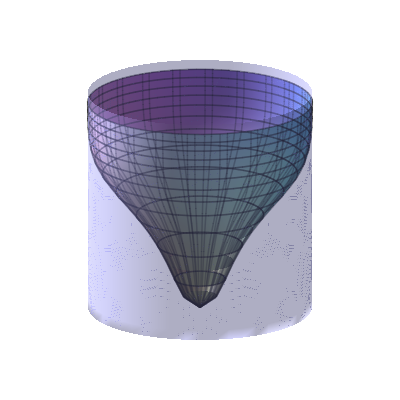}
        \caption{Singular rotational Bowl}
    \end{subfigure}
    \hfill
    \begin{subfigure}[t]{0.40\textwidth}
        \centering
        \includegraphics[width=\linewidth]{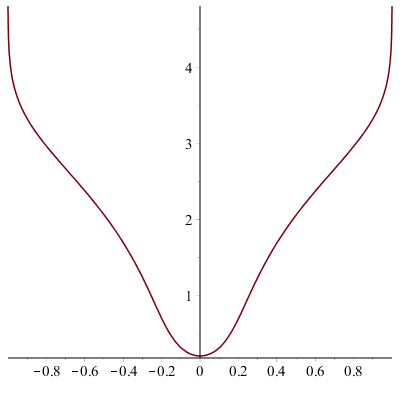}
        \caption{Profile curve}
    \end{subfigure}
     \caption{Singular rotational Bowl and its profile curve.}
\end{figure}

By the initial condition, we have that 
\begin{equation*}
    \lim_{t \to 0^+} \sinh x(t) =  0 \quad \text{and}\quad \lim_{t \to 0^+} \sin \theta(t)\cosh x(t) = 0. 
\end{equation*}

Note that in describing the $\varphi-$minimal Bowl, our analysis depended on the condition that the solution of (\ref{prob 1})-(\ref{initial cond 2}) is symmetric about the vertical axis. This symmetry occurs only when {$x(0) = 0$}. Then, we are interested in investigating the solutions of (\ref{prob 1}) under the initial condition $x(0) = x_0 \neq 0$. As we can see in \cite{bueno2018translating} and \cite{martin2015topology}, in a similar case, we obtain a Translating Catenoid. 

\subsection{Non-symmetric case}We want to describe the solution of (\ref{prob 1}) with the initial condition
\begin{align} \label{cond inicial nao convexo}
    x(0) = x_0 > 0, & &z(0) = z_0 \in (a, \infty), & &\theta(0) = 0,
\end{align}
in a negative time. In this case, observe that the solution is not symmetric about the vertical axis. 
Without loss of generality, we may restrict our analysis to $x_0> 0$,  as the case 
$x_0< 0$  (with $\theta(0) = 0$) exhibits identical behaviour up to reflection.

Then, we need to analyse what happens with the solution for positive and negative time. For positive time, as we proved in the previous section, the solution is a vertical, convex graph. However, for a negative time,  we define the functions $\tilde{x}(t) = x(-t)$, $\tilde{z}(t) = z(-t)$ and $\tilde{\theta}(t) = \theta(-t) + \pi$, which are solutions of (\ref{prob 1}) with the initial condition
\begin{equation*}
    \tilde{x}(0) = x_0 > 0, \quad \tilde{z}(0) = z_0 \quad \text{and} \quad \tilde{\theta}(0) = \pi.
\end{equation*}
Note that $\tilde{z}$ is an increasing function and if  $\tilde{x}$ has a critical point, it must be unique.

Now we proceed as in the parabolic case and we will study what happen in a negative time. Hence, consider $t<0$ and define the functions
\begin{equation*}
    \tilde{x}(t) = x(-t), \quad \tilde{z}(t) = z(-t), \quad  \text{and} \quad \tilde{\theta}(t)= \theta(-t)+\pi,
\end{equation*}
Note that $(\tilde{x}, \tilde{z}, \tilde{\theta})$ are solutions of the ODE system

\begin{equation} \label{ode parabolic negative time}
    \left\{\begin{array}{lll}
         \tilde{x}' & = &\cos \tilde{\theta}  \\
         \tilde{z}' & = &\sin \tilde{\theta} \\
         \tilde{\theta}' & = & \dot{\varphi}(\tilde{z}) \cos\tilde{\theta} - \coth \tilde{x} \sin \tilde{\theta},
    \end{array}\right.
\end{equation}
In the next proposition, we study the solutions of the system above under the following initial conditions:
\begin{equation}
    \tilde{x}(0) = x_0 > 0, \quad \tilde{z}(0) = z_0 \quad \text{and} \quad \tilde{\theta}(0) = \pi. \label{initial-cond-neg-rot}
\end{equation}

\begin{prop}\label{prop:rot-tempo-negativo-prop}
    Let $\tilde{x}$, $\tilde{z}$ and $\tilde{\theta}$ solutions of (\ref{ode parabolic negative time})-(\ref{initial-cond-neg-rot}). Then, the solution is defined for all positive time. Moreover, $\tilde{z}$ is increasing, $\tilde{x}$ has a unique critical point $t_0$, which is a minimum point and the following limits holds 
    \begin{equation*}
        \lim_{t\to \infty}\tilde{x}(t) = \infty, \quad \lim_{t\to \infty}\tilde{z}(t) = \infty, \quad\text{and} \quad 
        \lim_{t\to \infty} \tilde{\theta} = 0.
    \end{equation*}
\end{prop}
\begin{proof}
Arguing as in Lemma \ref{crescimento de z}, we show that $\tilde{z}$ is an increasing function in $(0,t_{\textnormal{max}})$, therefore $\tilde{\theta} \in (0, \pi)$ in this interval. Suppose by contradiction that $\tilde{\theta} \in (\pi/2, \pi)$. In this case, $\tilde{x}'<0$ and $\tilde{\theta}'<0$. Consequently $\tilde{x} \in (0,x_0)$ and, since both functions are bounded, there exist the limits 
\begin{equation*}
        \tilde{x}_L = \lim_{t \to {\textnormal{max}}} \tilde{x}(t), \quad \text{and} \quad \tilde{\theta}_L = \lim_{t \to {\textnormal{max}}} \tilde{\theta}(t).
    \end{equation*}
Moreover, since $\coth(\tilde{x})>1$ and $\cos(\tilde{\theta})<0$ we have  $-  {\tilde{\theta}'} > \tilde{z}'$, and the boundedness of ${\tilde{\theta}}$ implies that $\tilde{z}$ is bounded, therefore the limit $\tilde{z}_L = \lim_{t \to \textnormal{max}} \tilde{z}(t)$ exists.

We claim that $\tilde{x}_L > 0$. In fact, by Lemma \ref{y0=0}, if $\tilde{x}_L = 0$ then $\tilde{\theta}_L = \pi$, which is impossible since $\tilde{\theta}(0)=\pi$ and $\tilde{\theta}'<0$. Therefore $\tilde{x}_L > 0$ and it follows by Proposition \ref{maximal interval} that $t_\textnormal{max} = \infty$. 
Then the limit $\lim_{t \to \infty} \tilde{\theta}' (t) = \tilde{\theta}_L^1$ also exists. In this case, we have necessarily $\tilde{\theta}_L^1 = 0$, which is a contradiction since
$$
\tilde{\theta}_L^1 = \dot{\varphi}(\tilde{z}_L) \cos (\tilde{\theta}_L) - \sin (\tilde{\theta}_L) \coth(x_L)<0.
$$
We conclude that there exists $t_0 \in (0, t_{\textnormal{max}})$ such that $\tilde{\theta}(t_0) = \pi/2.$ Consequently, $t_0$ is a critical point for $\tilde{x}$. The same argument applied in Lemma \ref{crescimento de z} shows that this is the only critical point for $\tilde{x}$ and $\tilde{x}$ is an increasing function in $(t_0, t_{\textnormal{max}})$, which shows that $x_->0$ and, therefore, the solution is defined for all time and $\tilde{x}(t_0)$ is a local minimum. We claim $\tilde{x}(t_0)$ is a global minimum. In fact, using again that $\tilde{x}''(t_*)= \sin^2 \tilde{\theta}(t_*) \coth(\tilde{x}(t_*)>0$ at any possible critical point $t_*$, we conclude that $\tilde{x}$ has no critical point in $(0,\infty)$. As consequence, $\tilde{\theta} \in (0,\pi /2)$ in $(t_0, \infty)$. The rest of the proof follows now the same arguments as in Lemma \ref{lemma 1 rot}.
\end{proof}

\begin{theorem}\label{rotatinal winglike}
There exists a family of properly embedded $\varphi$-minimal surface in $\mathbb{H}^2\times\mathbb{R}$, rotationally symmetric. The surface has the topology of an annulus and is a bi-graph. Moreover, each of its graphical components has unbounded hyperbolic distance from the origin of $\mathbb{H}^2$, the height function is unbounded, and its generating curve is of winglike type.
\end{theorem}

\begin{figure}[ht]
    \centering    \begin{subfigure}[t]{0.50\textwidth}
        \centering
        \includegraphics[width=\linewidth]{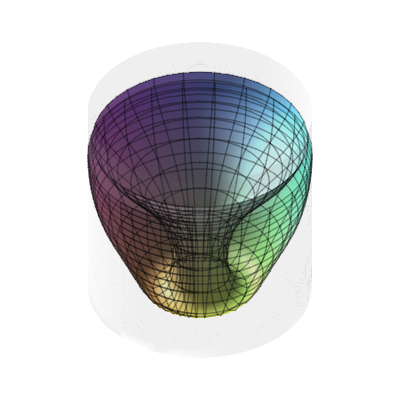}
        \caption{Singular rotational Winglike}
    \end{subfigure}
    \hfill
    \begin{subfigure}[t]{0.40\textwidth}
        \centering
        \includegraphics[width=\linewidth]{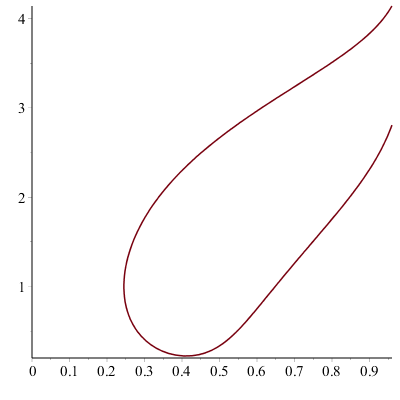}
        \caption{Profile curve}
    \end{subfigure}
     \caption{Singular rotational winglike and its profile curve.}
\end{figure}

These types of surface are known by winglike surfaces or, in the translating solitons case, the Translating Catenoids. In our work, we will call them $\varphi-$minimal winglike. Observe, as mentioned by Bueno in \cite{bueno2018translating}, the winglike are characterized as a one parameter family of immersed annuli $\{C_r\}$, where the parameter $r$ denotes the distance of each $\{C_r\}$ to the axis of rotation. Moreover, each $C_r$ lies outside of a cylinder of radius $r$.

\section{The \texorpdfstring{$\varphi$-minimal symmetric}{phi-minimal} surface without a minimum height}\label{sin minimo section}

This section is dedicated to studying a $\varphi$-minimal surface whose height function does not have a minimum point. The ideas here are similar to Proposition \ref{prop existencia}, where we analyse the $\varphi-$minimal rotational and $\varphi$-minimal parabolic near the singular point $z = a$. 

First, let $\Sigma$ be a $\varphi$-minimal surface foliated by geodesic circles (or horocycles). Consider $\Sigma$ as a horizontal graph given by $x = y(z)$, where $z$ is the parameter along the vertical direction. From Equation \eqref{system 1}, we have
\begin{equation*}
    \frac{\text{d}y}{\text{d} z} = \cot \theta.
\end{equation*}
Then, 
\begin{equation}\label{graph horizontal}
    y'' = \left( 1 +y'^2 \right) \left( (\tanh y)^i - \dot{\varphi}(z)y'\right),
\end{equation} 
with 
\begin{table}[ht]
\centering
\begin{tabular}{|l|l|l|}\hline
    $i$ & Function & Case \\ \hline
    $-1$ & $\coth(y)$ & Rotational\\ \hline
    $0$ & $-1$ & Parabolic \\ \hline
\end{tabular}
\end{table}

By a direct computation, equation \eqref{graph horizontal} is equivalent to
\begin{equation}\label{eqv}
    \left(
        \frac{y'}{\sqrt{1+y'^2}}e^{\varphi(z)}
    \right)'
    =
    e^{\varphi(z)}
    \frac{(\tanh y)^i}{\sqrt{1+y'^2}}.
\end{equation}
We will study the degenerate initial value problem given by
\eqref{eqv} and
\begin{equation} \label{cond inicial degenerada}
    y(a)=x_0,\qquad y'(a)=y_0.
\end{equation}

Define the functions $f_i: \mathbb{R}^2 \to \mathbb{R}$ and $g: \mathbb{R}\to \mathbb{R}$ as
\begin{equation*}
      f_i(s,t)=\frac{(\tanh s)^i}{\sqrt{1+t^2}},
\qquad
g(t)=\frac{t}{\sqrt{1+t^2}}
\end{equation*}

With this notation, the degenerate initial value problem becomes
\begin{equation}\label{degenerate system}
\begin{cases}
\displaystyle
\left(e^{\varphi(z)}g(y'(z))\right)'
    =e^{\varphi(z)}f_i(y(z),y'(z)), \\
y(a)=x_0,\quad y'(a)=y_0.
\end{cases}
\end{equation}
In the rotational case, we assume $x_0\neq 0$, since
$f_{-1}$ is not defined at $x_0=0$.

We first determine the admissible value of $y_0$ and then prove
existence and uniqueness.

\begin{lemma} Let $y=y(z)$, $z\in]a,a+\delta[$ be  a solution of \eqref{degenerate system} and suppose that there exist $\lim_{z\to a}y(z) = x_0\neq 0$ and  $\lim_{z\to a} y'(z)=L_0\in[-\infty,\infty]$. Then $L_0=0$.
\end{lemma}
\begin{proof}
    Fix $\delta > 0$ and let $a < \varepsilon < z < a+ \delta$.  Integrating \eqref{eqv} between $\varepsilon$ and $z$, we obtain
$$ e^{\varphi(z)}g(y'(z)) -
e^{\varphi(\varepsilon)}g(y'(\varepsilon))
=
\int_\varepsilon^z
e^{\varphi(t)}f_i(y(t),y'(t))\,dt.
$$
Since $g$ is bounded and $\varphi(\varepsilon)\to-\infty$ as
$\varepsilon\to a^+$, the second term on the left-hand side tends
to zero. Thus,
$$
g(y'(z))
=
e^{-\varphi(z)}
\int_a^z e^{\varphi(t)}f_i(y(t),y'(t))\,dt.
$$

The function
$f_i(y(t),y'(t))$ is bounded near $a$ and, for some $M>0$,
$$
|g(y'(z))|
\leq
M\frac{\displaystyle\int_a^z e^{\varphi(t)}\,dt}
        {e^{\varphi(z)}}.
$$
Applying L'Hôpital's rule and using the hypotheses on $\varphi$, we get
$$
\lim_{z\to a^+}
\frac{\displaystyle\int_a^z e^{\varphi(t)}\,dt}
     {e^{\varphi(z)}}
=
\lim_{z\to a^+}\frac{1}{\dot{\varphi}(z)}
=0.
$$
Therefore ${\lim_{z\to a}g(y'(z))=0}$. Since
$g(t)=t/\sqrt{1+t^2}$, we conclude that $L_0=0$.
\end{proof}

\subsection{Parabolic without minimum point} 

Consider the operator $I: C^1([a, a+\delta]) \to C^1([a, a+\delta])$, with $\delta$ to be determined, given by
\begin{equation*}
    I\left(y(s)\right) = e^{ - \varphi(s)}\int_a^s e^{\varphi(t)} f_0(y(t), y'(t)) dt.
\end{equation*}
Define the operator $S: C^1([a, a+\delta]) \to C^1([a, a+\delta])$ by
\begin{equation}
    S\left(y(z)\right) = \int_a^z g^{-1}\left(I\left(y(s)\right)\right)ds  + x_0 .
\end{equation}
We will prove that, there exists $\delta > 0$ such that $S$ is well defined and has a unique fixed point. Then, we will prove that \eqref{graph horizontal} has a solution arbitrarily close to $z = a$.

\begin{prop}\label{oper i}
    There exists $\delta > 0$ such that 
    $$|I(y(s))| < 1, \quad \forall s \in (a, a+\delta).$$
\end{prop}
\begin{proof}

By Remark \ref{implicações hipotese 3},
$\varphi(s)\to-\infty$ and $\dot{\varphi}(s)\to\infty$
as $s\to a^+$.  
Since $|f_0(y(t),y'(t))|\leq 1$, we have
$$
|I(y(s))|
\leq
e^{-\varphi(s)}\int_a^s e^{\varphi(t)}\,dt.
$$By L'Hôpital's rule,
$$
\lim_{s\to a^+}
e^{-\varphi(s)}\int_a^s e^{\varphi(t)}\,dt
=
\lim_{s\to a^+}\frac{1}{\dot{\varphi}(s)}
=0.
$$
Therefore, we may choose $\delta>0$ such that
$|I(y(s))|<1$ for every $s\in(a,a+\delta)$.
\end{proof}

\begin{prop}\label{contraction}
    The operator $S$ is a contraction
\end{prop}
\begin{proof}
    Let $\varepsilon > 0$ such that $f_0$ and $g^-1$ are Lipschitz in the domains {$[x_0-\varepsilon, x_0+\varepsilon]\times ]-\varepsilon,\varepsilon[$ and $[-\varepsilon, \varepsilon]$, respectively. Thus, for all $y, u \in B(x_0, \varepsilon)$ and for all $z \in (a, a+\delta)$}, we have
      \begin{align*}
        |S\left(y(z)\right) - S\left(u(z)\right)| &\leq c_1 \int_a^z \text{e}^{-\varphi(s)} \left|\int_a^s \text{e}^{\varphi(t)}|f(y,y')- f(u,u') dt \right|ds\\
        &\leq c_2\left( \left\lVert y - u \right \rVert_\infty + \left\lVert y' - u' \right\rVert_\infty\right)\int_a^z\int_a^s\text{e}^{-\varphi(s)}\text{e}^{\varphi(t)}dtds
    \end{align*}

        From Proposition \ref{oper i}, the integral
   $$\int_a^s\text{e}^{-\varphi(s)}\text{e}^{\varphi(t)}dt$$
   is bounded. Hence, we conclude that for $\delta >0$ small enough, $S$ is a contraction in the closed ball {$\overline{B(a,\varepsilon)}$} in $C^1[a,a+\delta])$.
\end{proof}

\begin{theorem}\label{parabolic locally non-mim}
There exists a unique solution $y \in C^2([a, a+\delta])$ of the integral equation
$$ y(z) = \int_a^z g^{-1}\!\left( e^{-\varphi(s)} \int_a^s f_0(y(t), y'(t)) e^{\varphi(t)} dt \right) ds + x_0. $$
Moreover, 
$$y'(a)=0,\qquad y''(a)=-\frac{1}{1+L}. $$
\end{theorem}
\begin{proof}
  By Proposition \ref{contraction} and the Banach fixed-point theorem, the integral equation has a unique solution in $C^1([a,a+\delta])$ and $y(a)=x_0$. Differentiating it for $z>a$, we see that $y$ solves the degenerate initial value problem \eqref{degenerate system}, with $y_0=y'(a)$. By the preceding lemma, $y'(a)=0$. 

  To study its second derivative at $a$, write
$$
g(y'(z))
=e^{-\varphi(z)}
\int_a^z e^{\varphi(t)}
f_0(y(t),y'(t))\,dt.
$$
Since $f_0(y(z),y'(z))\to-1$, L'Hôpital's rule yields
\begin{align*}
\lim_{z\to a^+}\frac{g(y'(z))}{z-a}
&=
\lim_{z\to a^+}
\frac{\displaystyle\int_a^z e^{\varphi(t)}
f_0(y(t),y'(t))\,dt}
{(z-a)e^{\varphi(z)}}\\
&=
\lim_{z\to a^+}
\frac{f_0(y(z),y'(z))}
{1+(z-a)\dot{\varphi}(z)}
=-\frac{1}{1+L}.
\end{align*}
On the other hand, 
\begin{align*}
    -\frac{1}{1+L} &= \lim_{z \to a^+} \frac{g(y')}{z-a} \\
            & = \lim_{z \to a^+} \frac{y'(z)}{z-a} = y''(a). 
\end{align*}
Thus, $y'$ is differentiable in $a$. The continuity of $y''$ in $a$ follows by a straight forward computation letting $z \to a$ in \eqref{graph horizontal}. Therefore, $y \in C^2([a, a+\delta])$ and $y''(a) = -\frac{1}{1+L}$.

\end{proof}

The previous theorem provides a local solution of
\eqref{graph horizontal} satisfying $y(a)=x_0$ and $y'(a)=0$.
We now extend this solution to its maximal interval of existence
and study the geometry of the corresponding parabolic
$\varphi$-minimal surface.

\begin{theorem} \label{teo sin min parabolic}
For every $x_0\in\mathbb R$, there exists a
$\varphi$-minimal surface $\Sigma$ in
$\mathbb H^2\times(a,\infty)$, invariant under a
one-parameter group of parabolic isometries of
$\mathbb H^2$ and regular up to its boundary, with
$(\cosh(x_0),0,\sinh(x_0),a)\in\partial\Sigma$.
The height function does not attain a minimum in the
interior of $\Sigma$. Moreover, both the distance
function and the height function are unbounded.
\end{theorem}
\begin{proof}

    The local existence of such a surface  is given by Theorem \ref{parabolic locally non-mim},  for a sufficiently small $\delta > 0$. We can extend this solution for the maximal interval of existence. Note that $y$ can be written as $x(t) = y\left(z(t)\right) $, where $x$ and $z$ are solutions of \eqref{ode parabolic 2}, with the initial conditions
    $$x(0) = x_0, \quad z(0) = z_0 \quad \text{and}\quad \theta(0) = \theta_0,$$
    such that $(x_0, z_0, \theta_0)$  are arbitrarily close to $(x_0, a, \frac{\pi}{2})$. We will prove that $x'<0$, $z' > 0$ and both $x$ and $z$ are unbounded.

   \noindent \textbf{Claim 1}: $\theta_0 > \frac{\pi}{2}.$
   
Since $f_0<0$, the integral equation gives
$$
g(y'(z))
=
e^{-\varphi(z)}
\int_a^z e^{\varphi(t)}f_0(y(t),y'(t))\,dt<0
$$
for $z>a$ sufficiently close to $a$. As $g$ preserves
the sign, we have $y'(z)<0$. Choosing the orientation
so that $z'>0$ and $0<\theta_0<\pi$, the relation
$$
y'(z_0)=\cot\theta_0
$$
implies that $\theta_0>\pi/2$

\noindent \textbf{Claim 2}: The function $x$ decreases, $z$ increases, the solution is defined for all positive time, and
    $$\lim_{t \to \infty} x(t) = -\infty, \quad \lim_{t \to \infty} z(t) = \infty.$$

Note that $x'(0) < 0$ and $z'(0) > 0$. Suppose, by contradiction, there exists $t_0 > 0$ such that $x'(t_0) = 0$. Then, $x''(t_0) \geq  0$. However, 
$$ x''(t_0) = -\sin^2 \theta(t_0) = -1 < 0.$$
Analogously, we prove that $z' > 0$. By Proposition \ref{crescimento de z}, the solution is defined for all positive time. Finally, the result follows from Lemma \ref{lema 3} and Proposition \ref{prop x comportamento}.
\end{proof}

The figure \ref{sin mim par} illustrates the surface described in the previous theorem.

\begin{figure}[ht]
  \centering
  \begin{minipage}[b]{0.45\textwidth}\includegraphics[scale=0.75]{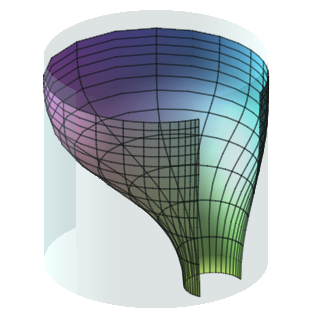} 
  \end{minipage}
  \hfill
  \begin{minipage}[b]{0.45\textwidth}\includegraphics[scale=0.5]{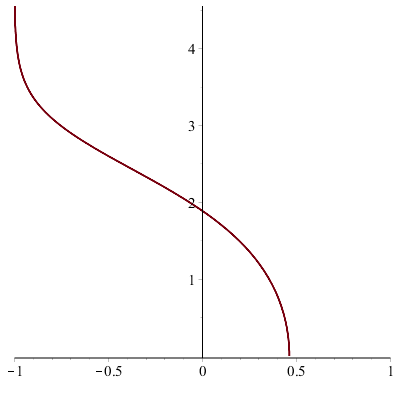}
  \end{minipage}
  \caption{$\varphi$-minimal for $\varphi(z) = \ln z$}\label{sin mim par}
\end{figure}

\subsection{Rotational surfaces without minimum at height}

Similar to the parabolic case, we define an operator $S$ and prove that it has a unique fixed point. This will be useful in proving the existence of a surface analogous to the one described in Theorem \ref{teo sin min parabolic}. However, in this case, we need to be careful since Equation \eqref{graph horizontal}, for $i = -1$, has a singularity at $y = 0$.

Let $\varepsilon > 0$. Define the set
$$\Omega_\varepsilon = \left\{ y \in C^1([a, a+\delta]) : |y(z)| \geq \varepsilon > 0, \ \forall z \in (a, a+\delta) \right\}.$$
Consider $I : \Omega_\varepsilon \to C^1([a, a + \delta])$ given by
$$I\left( y(s) \right)= e^{-\varphi(s)} \int_a^s e^{\varphi(t)} f_{-1}(y, y') \, dt.$$

The first step is to ensure that the operator $I$ is well defined and maps $\Omega_\varepsilon$ into a bounded set. This is necessary because $I$ appears in the definition of the operator $S_\varepsilon$, and we will need uniform bounds to apply the Banach fixed-point theorem later.

\begin{prop} \label{oper i-epsilon}
   There exists $\delta>0$ such that
$$
|I(y(s))|<1
$$
for every $y\in\Omega_\varepsilon$ and every
$s\in(a,a+\delta)$.
\end{prop}
\begin{proof}
    Since $|y|\geq\varepsilon$ for $y\in\Omega_\varepsilon$,
we have
$$
|f_{-1}(y,y')|
=\frac{|\coth y|}{\sqrt{1+y'^2}}
\leq\coth\varepsilon.
$$
Therefore,
$$
|I(y(s))|
\leq
\coth\varepsilon\,
e^{-\varphi(s)}\int_a^s e^{\varphi(t)}\,dt.
$$
Since $\varphi(s)\to-\infty$ as $s\to a^+$,
L'Hôpital's rule and the hypotheses on $\varphi$ give
$$
\lim_{s\to a^+}
e^{-\varphi(s)}\int_a^s e^{\varphi(t)}\,dt
=
\lim_{s\to a^+}\frac{1}{\dot{\varphi}(s)}
=0.
$$
Thus, choosing $\delta>0$ sufficiently small, we obtain
$|I(y(s))|<1$ for every $y\in\Omega_\varepsilon$ and
$s\in(a,a+\delta)$.
\end{proof}

Define $S_\varepsilon: \Omega_\varepsilon \to C^1[(a, a + \delta)]$ by
$$S_\varepsilon\left( y(z) \right) = x_0 + \int_a^z g^{-1}\left((I \left( y(s)\right)\right) \, ds,$$
where $|x_0| > \varepsilon$. Thus, every fixed point of
$S_\varepsilon$ satisfies $y(a)=x_0$.
The previous proposition ensures that the operator $S_\varepsilon$ is well defined.

However, for the fixed-point argument, we must also establish that $S_\varepsilon$ maps the set $\Omega_\varepsilon$ into itself. This is the content of the following proposition.

\begin{prop} \label{imagem s-epsilon}
    For $\delta>0$ sufficiently small,
$S_\varepsilon(\Omega_\varepsilon)\subset\Omega_\varepsilon$.
\end{prop}
\begin{proof}

By Proposition \ref{oper i-epsilon}, we may choose
$\delta>0$ such that $|I(y(s))|\leq1/2$ for every
$y\in\Omega_\varepsilon$ and $s\in[a,a+\delta]$.
Since $g^{-1}(t)=t/\sqrt{1-t^2}$, we obtain
$$
|S_\varepsilon(y(z))-x_0|
\leq\int_a^z
\left|g^{-1}\!\left(I(y(s))\right)\right|\,ds
\leq\frac{\delta}{\sqrt3}.
$$
Taking $\delta$ smaller if necessary so that
$|x_0|>\varepsilon+\delta/\sqrt3$, we conclude that
$$
|S_\varepsilon(y(z))|
\geq |x_0|-\frac{\delta}{\sqrt3}
>\varepsilon.
$$
Therefore, $S_\varepsilon(y)\in\Omega_\varepsilon$.

\end{proof}

 The strategy is to apply the Banach fixed-point theorem to the operator $S_\varepsilon$ restricted to the complete metric space $\Omega_\varepsilon$.  Proposition \ref{contraction} guarantees that $S_\varepsilon$ is a contraction, while Proposition \ref{imagem s-epsilon} ensures that it maps $\Omega_\varepsilon$ into itself.
 Therefore, the hypotheses of the fixed-point theorem are satisfied and the following theorem establishes the existence of a local solution of \eqref{graph horizontal}, $ i = -1$, near $z = a$.

\begin{theorem}\label{rot locally non-mim}
There exists a unique solution $y \in \Omega_\varepsilon$ of the integral equation
$$ y(z) = \int_a^z g^{-1}\!\left( e^{-\varphi(s)} \int_a^s f_{-1}(y(t), y'(t)) e^{\varphi(t)} dt \right) ds+ x_0, $$
where $|x_0|> \varepsilon$. Moreover, 
$$y'(a)=0,\qquad y''(a)=\frac{\coth x_0}{1+L}.$$
\end{theorem}
\begin{proof}
    By the same argument as in Proposition \ref{contraction},
$S_\varepsilon$ is a contraction for $\delta>0$ sufficiently
small. Since $\Omega_\varepsilon$ is complete and
$S_\varepsilon(\Omega_\varepsilon)\subset\Omega_\varepsilon$
by Proposition \ref{imagem s-epsilon}, the Banach fixed-point
theorem gives a unique fixed point $y\in\Omega_\varepsilon$.

The integral equation gives $y(a)=x_0$. Differentiating it
for $z>a$, we obtain
$$
g(y'(z))
=
e^{-\varphi(z)}
\int_a^z e^{\varphi(t)}
f_{-1}(y(t),y'(t))\,dt.
$$
Differentiating once more, we recover
\eqref{degenerate system}. By the preceding lemma,
$y'(a)=0$.

Since $y(a)=x_0\neq0$ and $y'(a)=0$, we have
$$
\lim_{z\to a^+}f_{-1}(y(z),y'(z))=\coth x_0.
$$
Thus, by L'Hôpital's rule,
\begin{align*}
\lim_{z\to a^+}\frac{g(y'(z))}{z-a}
&=
\lim_{z\to a^+}
\frac{\displaystyle\int_a^z e^{\varphi(t)}
f_{-1}(y(t),y'(t))\,dt}
{(z-a)e^{\varphi(z)}}\\
&=
\lim_{z\to a^+}
\frac{f_{-1}(y(z),y'(z))}
{1+(z-a)\dot{\varphi}(z)}
=\frac{\coth x_0}{1+L}.
\end{align*}
Finally, using the definition of $g$ and $y'(a)=0$,
we conclude that
\begin{align*}
  \frac{\coth x_0}{1 + L} &=\lim_{z\to a^+} \frac{g\left( y'(z) \right)}{z -a} \\
    & =   \lim_{z\to a^+}\frac{y'(z)}{z-a}  = y''(a).
\end{align*}

Therefore, $y'$ is differentiable at $a$ and
$y''(a)=\coth x_0/(1+L)$.
\end{proof}

\begin{theorem}\label{teo sin min rot}

For every $x_0\in\mathbb{R}\setminus\{0\}$, there exists a
rotationally invariant surface $\Sigma$ in
$\mathbb H^2\times(a,\infty)$, regular up to its boundary, with
$(\cosh(x_0),0,\sinh(x_0),a)\in\partial\Sigma$.
The height function does not attain a minimum in the
interior of $\Sigma$. Moreover, both the distance
function and the height function are unbounded.
\end{theorem}
\begin{proof}
   The local existence of such a surface is given by the previous theorem. The solution can then be extended to its maximal interval of existence. Now, writing $y(z) = x(s)$, where $x$ and $z$ are solutions of \eqref{system 1} with initial conditions
    \begin{equation*}
        x(0) = x_0 > 0, \quad z(0) = z_0, \quad \theta(0) = \theta_0,
    \end{equation*}
    where $z_0$ and $\theta_0$ are arbitrarily close to $a$ and $\frac{\pi}{2}$, respectively. We will prove that $\theta_0 < \frac{\pi}{2}$.

    Suppose, by contradiction, that $ \frac{\pi}{2} < \theta_ 0  < \pi$. In this case, $\theta$ increases near $t = 0$ and $\cos \theta_0 < 0$ and $\sin \theta_0 > 0$. However, by the equation \eqref{system 1},
    $$ \cos \theta_0 \dot{\varphi}(z) > \sin \theta_0 \coth x_0 > 0,$$
    which is a contradiction. Therefore, by Lemma \ref{lemma 1 rot},  the solution is defined for all positive time, both functions $x$ and $z$ increase and 
    $$ \lim_{t \to \infty} x(t) = \infty \quad\text{and} \quad \lim_{t \to \infty} z(t) = \infty.$$
\end{proof}

The Figure \ref{fig sin min rot} illustrates a $\varphi$-minimal surface described in Theorem \ref{teo sin min rot}.

\begin{figure}[ht]
    \centering
    \begin{subfigure}[b]{0.50\textwidth}     \centering
        \includegraphics[width=\textwidth]{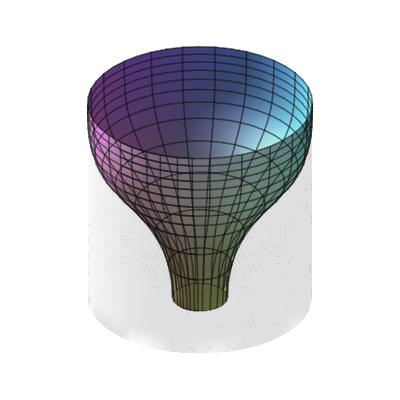} 
        \end{subfigure}
    \hfill 
    \begin{subfigure}[b]{0.40\textwidth} 
        \centering
        \includegraphics[width=\textwidth]{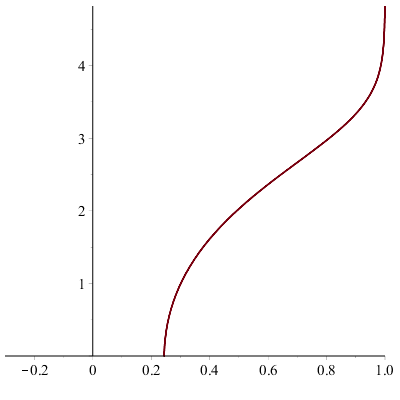} 
    \end{subfigure}
    \caption{$\varphi$-minimal for $\varphi(z) = \ln z$.}
    \label{fig sin min rot}
\end{figure}

\section{Classification of rotationally and parabolically invariant \texorpdfstring{$\varphi$-minimal}{phi-minimal} surfaces} \label{classification-section}

In the preceding sections, we constructed and analysed several families of $\varphi$-minimal surfaces invariant under parabolic translations or rotations of $\mathbb{H}^2$, namely: parabolic winglike surfaces (Theorem \ref{teo parabolic com minimo}), rotational Bowls (Theorem \ref{teo bowl rot}), rotational Winglike surfaces (Theorem \ref{rotatinal winglike}), and surfaces without a minimum point at height, both parabolic (Theorem \ref{teo sin min parabolic}) and rotational (Theorem \ref{teo sin min rot}). We now show that these examples exhaust all possible symmetric $\varphi$-minimal surfaces, invariant under parabolic translation and rotation, and under the hypotheses on $\varphi$ stated in Section \ref{preliminar}.

The following theorem is the main classification result of this work.

\begin{theorem}[Classification in the parabolic and rotational cases]
\label{teo:classificacao}
Let $\varphi$ be a connected, properly embedded $\varphi$-minimal surface in $\mathbb{H}^2\times\mathbb{R}$,  invariant under either a one-parameter group of parabolic translations or a one-parameter group of rotations of
$\mathbb{H}^2$, where $\varphi : (a,\infty) \to \mathbb{R}$ satisfies $\dot{\varphi} > 0$ and $\ddot{\varphi} < 0$. Then, $\Sigma$ is, up to an ambient isometry, one of the following:
\begin{itemize}
    \item a parabolic winglike surface;
    \item a parabolic horizontal graph: invariant under parabolic translations without a minimum point at height (Theorem \ref{teo sin min parabolic});
    \item a rotational bowl (Theorem \ref{teo bowl rot});
    \item a rotational winglike surface(Theorem \ref{rotatinal winglike});
    \item a rotational horizontal graph: invariant under rotations, without a minimum point at height (Theorem \ref{teo sin min rot}). 
\end{itemize}
Each of the surfaces listed above exists and is properly embedded, complete, and satisfies the $\varphi$-minimal surface equation.
\end{theorem}

\begin{proof}
The existence of each family follows from Theorems \ref{teo parabolic com minimo}, \ref{teo sin min parabolic}, \ref{teo bowl rot}, \ref{rotatinal winglike}, and \ref{teo sin min rot}. To prove that these are the only possibilities, observe that a symmetric $\varphi$-minimal surface is completely determined by its profile curve, which is a solution of the ODE system \eqref{system 1}. The qualitative analysis of these solutions, carried out in Sections \ref{parabolic section} through \ref{sin minimo section}, shows that the height function $z$ either attains a minimum point or does not. If it does, the surface is of bowl- or wing-like type, according to whether the profile curve intersects the axis of symmetry or not. If it does not, the surface is a horizontal graph defined on a horocycle or a geodesic cycle, depending on the symmetry. This exhausts all possible behaviours of the solutions of \eqref{system 1}, for $i = 0, -1$, and hence all symmetric $\varphi$-minimal surfaces.
\end{proof}

\section{Non-existence theorems for \texorpdfstring{$\varphi$-minimal symmetric}{phi-minimal} surfaces} \label{non existence section}

We now establish several non-existence results for $\varphi$-minimal surfaces in $\mathbb{H}^2 \times \mathbb{R}$. These results are inspired by analogous theorems for translating solitons in Riemannian products \cite{de2019translating} and in $\mathbb{H}^2 \times \mathbb{R}$ \cite{bueno2018translating}. They show that, under mild geometric constraints, no $\varphi$-minimal surface can exist.

We show that a properly immersed $\varphi$-minimal surface cannot be contained in a vertical cylinder. The proof uses the rotational winglike surfaces constructed in Theorem \ref{rotatinal winglike} as barriers and the tangency principle (Theorem \ref{thm:tangency-principle}).

\begin{theorem} \label{teo no compact into a cylinder}
    There exists no properly immersed $\varphi$-minimal surface in $\mathbb{H}^2\times\mathbb{R}$ that is contained within a vertical cylinder $B \times \mathbb{R}$, where $B$ is a closed curve in $\mathbb{H}^2$.
\end{theorem}
\begin{proof}
Assume, by contradiction, that 
the $\varphi$-minimal surface $M$ is contained in a vertical cylinder of radius $r_0$. Consider a family of rotational winglike  $\{\mathcal{C}_{r}\}$, with $r > r_0$, described in Theorem \ref{rotatinal winglike}. For every $r>r_0$, each $\mathcal{C}_r$ lies in the non-compact region of $\mathbb{H}^2\times\mathbb{R}$.
We now decrease $r$ to $r_1 > r_0$, where $r_1$ is the smallest radius for which $\mathcal{C}_r$  becomes tangent to $M$ at an interior point. By the tangency principle, $M$
 must coincide with $\mathcal{C}_{r_1}$. However, this contradicts the assumption that 
$M$ is contained inside the cylinder.
\end{proof}

Moreover, analogously to Theorem 14 of \cite{de2019translating}, we prove the following result with a similar strategy, but now using the parabolic winglike as barriers. 
\begin{theorem}\label{teo no compact into a horocylinder}
    There are no complete $\varphi$-minimal surface with compact  boundary, properly embedded in a solid horocylinder in $\mathbb{H}^2\times \mathbb{R}$.
\end{theorem}

For the next result, we need to establish some relevant equations, which were first proven in \cite{martinez2022equilibrium} for $\varphi-$minimal surfaces in the Euclidean space. First, define the angle function by
$$\nu(p) = \langle \eta, \partial_z\rangle(p),\quad \forall p \in \Sigma.$$
As is well known, $\partial_z$ is a vertical, parallel vector field. Moreover, we can decompose it as
$$ \partial_z = \nabla h +  \nu \eta,$$
where $h : \Sigma \rightarrow \mathbb{R}$ is the height function with respect to $\mathbb{R}$. Since $\Sigma$ is a $\varphi$-minimal surface, we have
\begin{equation}\label{eq1}
    1 = ||\partial_z||^2 = ||\nabla h||^2+ \left(\frac{H}{\dot{\varphi}}\right)^2
\end{equation}
On the other hand, from the definition of the Hessian operator, for any tangent vectors $X$ and $Y$  in $\Sigma$, we get
\begin{align*}
    \nabla^2 h (X,Y) & = \langle \nabla_X \left(\partial_z\right)^T, Y \rangle \\
    & = \langle \widetilde{\nabla}_X\partial_z + \widetilde{\nabla}_X (\partial_z)^{\perp}, Y \rangle\\
    & = \langle \mathbf{A}(X,Y), \partial_z\rangle,
\end{align*}
where $\widetilde{\nabla}$ denotes the Riemannian connection and $\mathbf{A}$ is the second fundamental form.  Consequently,
\begin{equation}\label{eq2}
    \dot{\varphi} \nabla^2 h = \mathcal{A}H,
\end{equation}
where $\mathcal{A}(X,Y) = \langle \mathbf{A}(X,Y), \eta \rangle$.

Equation \eqref{eq2} provides interesting information on the height function of a $\varphi$-minimal surface $\Sigma$.

\begin{prop}
    If $\varphi: (a, \infty) \to \mathbb{R}$ satisfies $\dot{\varphi}>0$ on $(a,\infty)$, then the height function $h$ of a $\varphi$-minimal surface $\Sigma$ cannot attain a local maximum at any interior point.
\end{prop}
\begin{proof}
    Let $\{E_i\}_i$ be a orthogonal frame of $T\Sigma$, then, by equations (\ref{eq1}) and (\ref{eq2}) we obtain that
    $$\Delta h = \sum_k \nabla^2(E_k, E_k) = \dot{\varphi} \left(1-||\nabla h ||^2\right)>0$$
    and the result follows. 
\end{proof}

Finally, we prove that there are no closed (compact without boundary) $\varphi$-minimal surfaces in $\mathbb{H}^2 \times \mathbb{R}$. This generalizes the corresponding result for translating solitons obtained in \cite{de2019translating}.

\begin{prop} \label{compact}
    If $\varphi: (a, \infty) \to \mathbb{R}$ satisfies $\dot{\varphi}>0$ on $(a,\infty)$, there does not exist a closed $\varphi$-minimal surface in $\mathbb{H}^2\times \mathbb{R}$.
\end{prop}
\begin{proof}
    Suppose by contradiction that there exists a $\varphi-$minimal surface $\Sigma$ in $\mathbb{H}^2\times\mathbb{R}$ compact without an empty boundary. Let $h : \Sigma \to \mathbb{R}$ be its height function. By the divergence theorem, we have
     \begin{eqnarray}
    \int_\Sigma \Delta h = \int_{\partial \Sigma } \langle \nabla h, \eta \rangle  = 0.
  \end{eqnarray}
  On the other hand,
  $$\Delta h = \dot{\varphi} \left(1-||\nabla h ||^2\right) = \dot{\varphi} \nu^2.$$
  Then, $\nu = \langle \eta, \partial_z \rangle = 0$ and $\Sigma$ is contained in a vertical plane, contradicting $\Sigma$ being closed.
\end{proof}

\subsubsection*{Acknowledgments and Funding} Katianny Rosa Freitas was supported by CAPES Finance Code 001, and she thanks the Departamento de Geometría y Topología at Universidad de Granada for the hospitality, where part of this work was conducted. João Paulo dos Santos was partially supported by CNPq 313200/2025-4 and FAPDF 0019300001678/2024-39. Antonio Martínez was partially supported by Ministerio de Ciencia,  Innovaci\'on y Universidades
Grant No: PID2024-160586NB-100; “Maria de Maeztu” Excellence Unit IMAG, reference CEX2020-
001105-M and Junta de Andaluc\'ia Grant No. FQM325.

\newpage
\bibliographystyle{plain}
\bibliography{sample}

@book{anderson2005hyperbolic,
  title={Hyperbolic Geometry},
  author={Anderson, James W.},
  series={Springer Undergraduate Mathematics Series},
  edition={2nd},
  publisher={Springer},
  address={London},
  year={2005},
  doi={10.1007/978-1-84800-000-0},
  isbn={978-1-85233-934-0}
}

@article{arezzo2013conformal,
  title={Conformal solitons to the mean curvature flow and minimal submanifolds},
  author={Arezzo, Claudio and Sun, J},
  journal={Mathematische Nachrichten},
  volume={286},
  number={8-9},
  pages={772--790},
  year={2013},
  publisher={Wiley Online Library}
}

@article{bueno2018translating,
  title={Translating solitons of the mean curvature flow in the space $\mathbb{H}^2\times\mathbb{R}$},
  author={Bueno, Antonio},
  journal={Journal of Geometry},
  volume={109},
  pages={1--26},
  year={2018},
  publisher={Springer}
}

@article{bueno2020uniqueness,
  title={Uniqueness of the translating bowl in $\mathbb{H}^2 \times \mathbb{R}$},
  author={Bueno, Antonio},
  journal={Journal of Geometry},
  volume={111},
  number={3},
  pages={1--7},
  year={2020},
  publisher={Springer}
}

@article{rosales2,
  AUTHOR={Castro, Katherine and Rosales, C\'{e}sar},
  TITLE={Uniqueness results and enclosure properties for hypersurfaces
         with boundary in weighted cylinders},
  JOURNAL={J. Math. Anal. Appl.},
  FJOURNAL={Journal of Mathematical Analysis and Applications},
  VOLUME={512},
  YEAR={2022},
  NUMBER={1},
  PAGES={Paper No. 126143, 19},
  ISSN={0022-247X,1096-0813},
  MRCLASS={53C40 (53C42)},
  MRNUMBER={4392290},
  MRREVIEWER={P.\ G.\ Walczak},
  DOI={10.1016/j.jmaa.2022.126143},
  URL={https://doi.org/10.1016/j.jmaa.2022.126143}
}

@article{clutterbuck2007stability,
  title={Stability of translating solutions to mean curvature flow},
  author={Clutterbuck, Julie and Schn{\"u}rer, Oliver C and Schulze, Felix},
  journal={Calculus of Variations and Partial Differential Equations},
  volume={29},
  pages={281--293},
  year={2007},
  publisher={Springer}
}

@article{de2020hypersurfaces,
  AUTHOR={de Lima, R. F. and Manfio, F. and dos Santos, J. P.},
  TITLE={Hypersurfaces of constant higher-order mean curvature in
         {$M\times\mathbb{R} $}},
  JOURNAL={Ann. Mat. Pura Appl. (4)},
  FJOURNAL={Annali di Matematica Pura ed Applicata. Series IV},
  VOLUME={201},
  YEAR={2022},
  NUMBER={6},
  PAGES={2979--3028},
  ISSN={0373-3114,1618-1891},
  MRCLASS={53B25 (53C24 53C42)},
  MRNUMBER={4498418},
  DOI={10.1007/s10231-022-01229-3},
  URL={https://doi.org/10.1007/s10231-022-01229-3}
}

@article{lima-pipoli,
  AUTHOR={de Lima, Ronaldo F. and Pipoli, Giuseppe},
  TITLE={Translators to higher order mean curvature flows in {$\mathbb{R}^n\times\mathbb{R}$} and {$\mathbb{H}^n\times\mathbb{R}$}},
  JOURNAL={J. Geom. Anal.},
  FJOURNAL={Journal of Geometric Analysis},
  VOLUME={35},
  YEAR={2025},
  NUMBER={3},
  PAGES={Paper No. 92, 50},
  ISSN={1050-6926,1559-002X},
  MRCLASS={53E40 (53C42 53C45)},
  MRNUMBER={4863897},
  MRREVIEWER={Peter\ Olamide\ Olanipekun},
  DOI={10.1007/s12220-025-01925-5},
  URL={https://doi.org/10.1007/s12220-025-01925-5}
}

@article{de2019translating,
  title={Translating solitons in Riemannian products},
  author={de Lira, Jorge and Martín, Francisco},
  journal={Journal of Differential Equations},
  volume={266},
  number={12},
  pages={7780--7812},
  year={2019},
  publisher={Elsevier}
}

@article{dierkes2021symmetric,
  title={Symmetric solutions of the singular minimal surface equation},
  author={Dierkes, Ulrich and Groh, Nico},
  journal={Annals of Global Analysis and Geometry},
  volume={60},
  number={2},
  pages={431--453},
  year={2021},
  publisher={Springer}
}

@article{joeri,
  AUTHOR={Dillen, Franki and Fastenakels, Johan and Van der Veken,
          Joeri},
  TITLE={Rotation hypersurfaces in {$\mathbb{S}^n\times\mathbb{R}$} and {$\mathbb{H}^n\times\mathbb{R}$}},
  JOURNAL={Note Mat.},
  FJOURNAL={Note di Matematica},
  VOLUME={29},
  YEAR={2009},
  NUMBER={1},
  PAGES={41--54},
  ISSN={1123-2536,1590-0932},
  ISBN={978-88-548-3482-8},
  MRCLASS={53C40 (53C50)},
  MRNUMBER={2779904},
  MRREVIEWER={Ivko\ Dimitric}
}

@article{eschenburg1989maximum,
  title={Maximum principle for hypersurfaces},
  author={Eschenburg, J-H},
  journal={manuscripta mathematica},
  volume={64},
  number={1},
  pages={55--75},
  year={1989},
  publisher={Springer}
}

@misc{tarcios-joao,
  author={Ferreira, T. A. and dos Santos, J. P.},
  title={Classification of minimal surfaces and solitons to the mean curvature flow in $\mathbb{H}^3$ as translation surfaces},
  year={2025},
  note={arXiv:2511.21545v1},
  url={https://arxiv.org/abs/2511.21545v3}
}

@misc{hyperbolic-FMS,
  author={Freitas, K. F. and Mart\'inez, A. and dos Santos, J. P.},
  title={Weighted minimal surfaces in $\mathbb{H}^2 \times \mathbb{R}$ invariant under hyperbolic translations},
  year={2026},
  note={In preparation}
}

@inproceedings{hoffman2017notes,
  title={Notes on translating solitons for mean curvature flow},
  author={Hoffman, David and Ilmanen, Tom and Mart{\'i}n, Francisco and White, Brian},
  booktitle={m: iv Workshops},
  pages={147--168},
  year={2017},
  organization={Springer}
}

@article{rosales1,
  AUTHOR={Hurtado, A. and Palmer, V. and Rosales, C.},
  TITLE={Parabolicity criteria and characterization results for
         submanifolds of bounded mean curvature in model manifolds with
         weights},
  JOURNAL={Nonlinear Anal.},
  FJOURNAL={Nonlinear Analysis. Theory, Methods \& Applications. An
            International Multidisciplinary Journal},
  VOLUME={192},
  YEAR={2020},
  PAGES={111681, 32},
  ISSN={0362-546X,1873-5215},
  MRCLASS={53C42 (31C12 35R01)},
  MRNUMBER={4031157},
  MRREVIEWER={S.\ Carolina\ Garc\'{\i}a-Mart\'{\i}nez},
  DOI={10.1016/j.na.2019.111681},
  URL={https://doi.org/10.1016/j.na.2019.111681}
}

@article{lopez2018invariant,
  title={Invariant singular minimal surfaces},
  author={L{\'o}pez, Rafael},
  journal={Annals of Global Analysis and Geometry},
  volume={53},
  number={4},
  pages={521--541},
  year={2018},
  publisher={Springer}
}

@article{lopez2019dirichlet,
  title={The Dirichlet problem for the $\alpha$-singular minimal surface equation},
  author={L{\'o}pez, Rafael},
  journal={Archiv der Mathematik},
  volume={112},
  number={2},
  pages={213--222},
  year={2019},
  publisher={Springer}
}

@article{ManfioTojeiro2011,
  author={Manfio, Fernando and Tojeiro, Ruy},
  title={Hypersurfaces with constant sectional curvature of
         \(\mathbb{S}^n \times \mathbb{R}\) and \(\mathbb{H}^n \times \mathbb{R}\)},
  journal={Illinois Journal of Mathematics},
  volume={55},
  number={1},
  pages={397--415},
  year={2011},
  mrnumber={MR2856874},
  zblnumber={1245.53021},
  url={https://projecteuclid.org/euclid.ijm/1359766951}
}

@article{martin2015topology,
  title={On the topology of translating solitons of the mean curvature flow},
  author={Mart{\'i}n, Francisco and Savas-Halilaj, Andreas and Smoczyk, Knut},
  journal={Calculus of Variations and Partial Differential Equations},
  volume={54},
  number={3},
  pages={2853--2882},
  year={2015},
  publisher={Springer}
}

@article{weierstrass-martinez-trivino,
  AUTHOR={Mart\'{\i}nez, Antonio and Mart\'{\i}nez-Trivi\~{n}o, A. L.},
  TITLE={A {W}eierstrass type representation for translating solitons
         and singular minimal surfaces},
  JOURNAL={J. Math. Anal. Appl.},
  FJOURNAL={Journal of Mathematical Analysis and Applications},
  VOLUME={516},
  YEAR={2022},
  NUMBER={2},
  PAGES={Paper No. 126528, 11},
  ISSN={0022-247X,1096-0813},
  MRCLASS={53A10 (35C08)},
  MRNUMBER={4457891},
  MRREVIEWER={Pradip\ Kumar},
  DOI={10.1016/j.jmaa.2022.126528},
  URL={https://doi.org/10.1016/j.jmaa.2022.126528}
}

@article{mean-convex,
  AUTHOR={Mart\'{\i}nez, Antonio and Mart\'{\i}nez-Trivi\~{n}o, Antonio
          Luis and dos Santos, Jo\~{a}o Paulo},
  TITLE={Mean convex properly embedded
         {$[\varphi, \vec{e}_3]$}-minimal surfaces in
         {$\mathbb{R}^3$}},
  JOURNAL={Rev. Mat. Iberoam.},
  FJOURNAL={Revista Matem\'{a}tica Iberoamericana},
  VOLUME={38},
  YEAR={2022},
  NUMBER={4},
  PAGES={1349--1370},
  ISSN={0213-2230,2235-0616},
  MRCLASS={35Q53 (53C42)},
  MRNUMBER={4445917},
  DOI={10.4171/rmi/1352},
  URL={https://doi.org/10.4171/rmi/1352}
}

@article{martinez2022equilibrium,
  title={Equilibrium of surfaces in a vertical force field},
  author={Mart{\'\i}nez, Antonio and Mart{\'\i}nez-Trivi{\~n}o, AL},
  journal={Mediterranean Journal of Mathematics},
  volume={19},
  number={1},
  pages={1--28},
  year={2022},
  publisher={Springer}
}

@article{catenary-cylinders,
  AUTHOR={Mart\'{\i}nez-Trivi\~{n}o, A. L. and dos Santos, J. P.},
  TITLE={Uniqueness of the {$[\varphi, \overrightarrow e_3]$}-catenary
         cylinders by their asymptotic behaviour},
  JOURNAL={J. Math. Anal. Appl.},
  FJOURNAL={Journal of Mathematical Analysis and Applications},
  VOLUME={514},
  YEAR={2022},
  NUMBER={2},
  PAGES={Paper No. 126347, 17},
  ISSN={0022-247X,1096-0813},
  MRCLASS={53A10 (35B33)},
  MRNUMBER={4429585},
  MRREVIEWER={Ergin\ Bayram},
  DOI={10.1016/j.jmaa.2022.126347},
  URL={https://doi.org/10.1016/j.jmaa.2022.126347}
}

@misc{half-space,
  author={Mart\'{\i}nez-Trivi\~{n}o, A. L. and dos Santos, J. P. and Tinaglia, T.},
  title={Half-space type theorems for a class of weighted minimal surfaces in $\mathbb{R}^3$},
  year={2026},
  note={arXiv:2601.22092},
  url={https://arxiv.org/abs/2601.22092}
}

@book{perko2013differential,
  title={Differential equations and dynamical systems},
  author={Perko, Lawrence},
  volume={7},
  year={2013},
  publisher={Springer Science \& Business Media}
}

@misc{tinaglia-pipoli,
  author={Pipoli, G. and dos Santos, J. P. and Tinaglia, T.},
  title={On the geometry of the asymptotic boundary of translators in $\mathbb{H}^2 \times \mathbb{R}$},
  year={2026},
  note={arXiv:2505.21083},
  url={https://arxiv.org/abs/2505.21083}
}

@article{poisson1975surfaces,
  title={Sur les surfaces elastique},
  author={Poisson, SD},
  journal={Men. CL. Sci. Math. Phys. Inst. Frace, deux},
  pages={167--225},
  year={1975}
}

@article{serrin-removesingular,
  title={Removable singularities of solutions of elliptic equations. II},
  author={James Serrin},
  journal={EBSCOhost Academic Search Premier},
  year={1965},
  volume={20},
  number={3},
  pages={163--169},
  doi={10.1007/bf00276442}
}

\end{document}